\documentclass[a4paper,12pt]{article}
\usepackage[utf8]{inputenc}
  \usepackage[usenames,dvipsnames]{xcolor}
\usepackage{amsmath,mathrsfs}
\usepackage{amsfonts}
  \usepackage[all]{xy}
  \usepackage[english]{babel}
  \usepackage{graphicx}
  \usepackage{enumerate}
  \usepackage{wrapfig}
\usepackage{amsthm}
\usepackage{hyperref}
\newtheorem{thm}{Theorem}    

\newtheorem{lem}{Lemma}[section]          
\newtheorem{prop}[lem]{Proposition}
\newtheorem{cor}[lem]{Corollary}
\newtheorem{example}[lem]{Example}
\newtheorem{defi}[lem]{Definition}    
\newtheorem{rem}{Remark}             

  \def\ZZ{\mathbb{Z}}
  \def\r{\mathfrak{r}}
  \def\D{\mathcal{D}}
  
  \def\H{\mathbb{H}}

  \def\U{\mathcal U}
  
  \def\Hom{\mathrm{Hom}}
  \def\GG{\SO_0(1,2)}
  \def\isom{Isom}
  \def\lsem{[\![}
  \def\rsem{]\!]}
\def\V{\mathcal V}

\def\A{\mathcal A}

\def\E{\mathbf E}

\def\N{\mathbb N}
\def\T{\mathcal{T}}

\def\d{\mathrm{d}}

\def\D{\mathcal{D}}
\def\SO{\mathrm{SO}}
\def\E{\mathbb{E}}
\def\BTZ{{\E_{0}^{1,2}}}
\def\sing{\mathrm{Sing}}
\def\RR{\mathbb{R}}
\newcommand{\mass}[1]{{\mathbb E}^{1,2}_{#1}}
\def\isom{\mathrm{Isom}}

\def\Reg{\mathrm{Reg}}
\def\fix{\mathrm{Fix}}

\def\reg{\Reg}

\def\V{\mathcal V}

\def\T{\mathcal{T}}

\def\W{\mathcal W} 

\def\SO{\mathrm{SO}}

\def\RR{\mathbb R}

\def\H{\mathbb{H}}

\def\E{\mathbb{E}}
\def\sing{\mathrm{Sing}}
\def\isom{\mathrm{Isom}}
\def\d{\mathrm{d}}

\def\N{\mathbb{N}}
\def\reg{\mathrm{Reg}}
\def\D{\mathcal{D}}
\def\BTZ{\mathbb{E}^{1,2}_0}
\def\U{\mathcal{U}}
\def\fix{\mathrm{Fix}}

\def\Reg{\mathrm{Reg}}

\def\r{\mathfrak{r}}

\def\fix{\mathrm{Fix}}

\def\sing{\mathrm{Sing}}

\def\Hom{\mathrm{Hom}}

\def\A{\mathcal A}
\def\B{\mathcal B}
\def\GG{\SO_0(1,2)}

\def\ord{\mathrm{Ord}}

\newcommand{\fonction}[5]{\displaystyle#1:\begin{array}{l|rcl}
& \displaystyle #2 & \longrightarrow & \displaystyle #3 \\
    & \displaystyle #4 & \longmapsto & \displaystyle #5 \end{array}}
\newcommand{\fonctiontrois}[9]{\displaystyle#1:\begin{array}{l|rcl}
& \displaystyle #2 & \longrightarrow & \displaystyle #3 \\
    & \displaystyle #4 & \longmapsto & \displaystyle #5\\
    & \displaystyle #6 & \longmapsto & \displaystyle #7\\
    & \displaystyle #8 & \longmapsto & \displaystyle #9\end{array}}

\definecolor{MyGreen}{rgb}{0.0,.5,0.0}

\definecolor{MyDarkRed}{rgb}{0.7,0,0}

\usepackage{float}

\title{On branched covering of singular (G,X)-manifolds}
\author{L\'eo Brunswic}

\def\CatSpread{\mathbf{Spr}}
\def\CatSingGX{\mathbf{SingGX}}

\begin{document}

\maketitle
\begin{abstract}    
Branched covering have a long history from ramification of Riemann surfaces to realization of 3 manifolds as covering ramified over a knots; from geometrical topology to algebraic geometry. The present work investigates a notion of branched covering "à la Fox" which is particularly natural for (G,X)-manifolds. The work is two fold. First, we recall and enrich the current state of the art (based upon Montesinos) and develop a Galois theory for such branched covering together with a description  of the fiber above branching points. As a consequence, we solve two open questions of Montesinos and construct an example related to another open question.  
Second, we present a theory of singular (G,X)-manifolds and apply the theory of branched covering we developped to extend the usual framework of (G,X)-manifolds to singular (G,X)-manifolds; in particular, we construct a developping map for such singular manifolds. An application to  singular locally Minkowski manifolds is given.
\end{abstract}

\tableofcontents

        \section{Introduction}

Consider $\H$ the hyperbolic plane and $\GG$ its group of direct isometries. The quotient of $\H$ by a geometrically finite discrete without torsion subgroup of $\GG$, say $\Gamma$, gives a locally hyperbolic manifold, ie  a $\H$-manifold, say $\Sigma:=\Gamma\backslash\H$. Such  a $\Sigma$ admits a developping map $\D:\widetilde \Sigma\rightarrow \H$ and an holonomy homomorphism
$\rho:\pi_1(\Sigma)\rightarrow \GG$ where $\widetilde \Sigma$ is the universal covering of $\Sigma$ and $\pi_1(\Sigma)$ its fundamental group. If furthermore the holonomies of the ends of $\Sigma$ are parabolic then one can associate to each ends an ideal point "at infinity", effectively compactifying $\Sigma$ to some $\overline \Sigma$ (which is the Freudenthal compactification \cite{MR931673} of $\Sigma$). The manipulation of such points "at infinity" (or even irrational conical singularities), though legal and quite simple in this specific context, require some additional care in the general setting of $(G,X)$-manifolds especially if one consider non-metric analytical structures such as conformal structures, Lorentzian structures or affine structures ie $(Aff(\RR^n),\RR^n)$-structures. Compactifications and completions are not $(G,X)$-manifolds anymore but singular $(G,X)$-manifolds. These manifolds lack the fundamental properties that make the strength of $(G,X)$-manifolds: beyong elementary analyticity statements we need a good notion of developping map, holonomy and universal covering.  

A natural first step in this direction is to consider, for some singular $(G,X)$-manifold $M$, the universal covering $\widetilde\reg(M)$ of its regular locus $\reg(M)$, which give rise to a map $\widetilde\reg(M)\rightarrow M$, then try to complete $\widetilde\reg(M)$ in a natural way.  The solution has been provided by Fox in the 1950's.
Recall that a topological space is $T_1$ if the intersection of  all neighborhoods of a given point is a singleton. Fox \cite{MR0123298,MR1152367} introduced the notion of {\it spread}: let $(X,\T_X)$ and $(Y,\T_Y)$ be two locally connected $T_1$ topological spaces a continuous map $X\xrightarrow{p} Y$ is a spread if the topology of $X$ is generated by the connected components of $p^{-1}(\U)$ when $\U$ go through the open subsets of $Y$. In the following, all topological spaces are assumed at least $T_1$ and locally connected. A point  $x\in X$ is then {\it ordinary} if there exists a neighborhood $\U$ of $p(x)$ such that $p$ maps each connected component of $p^{-1}(\U)$ homeomorphically onto $\U$; the set of ordinary point of a spread is denoted by $\ord(X\xrightarrow{p}Y)$ or by abuse of notation $\ord(X)$. Notice that a spread $X\rightarrow Y$ is a covering iff $\ord(X)=X$ and $X$ is connected.  For our purpose, the main result of Fox is the existence, uniqueness and fonctoriality of the completion of spreads. A spread is {\it complete } if for every $y\in Y$ and every increasing choice $\U\mapsto \hat \U$ of a connected component of $p^{-1}(\U)$ for $\U$ connected open neighborhood of $y$, then $\bigcap_{\U}\hat \U\neq \emptyset$.  
Given a spread $X\rightarrow Y$, a spread $\overline X \xrightarrow{\overline p}Y$ together with an embedding $X\xrightarrow{\iota} \overline X$ lifting the identity $Y\rightarrow Y$ is a {\it completion} of $X\rightarrow Y$ if it is complete and if the image of $X$ is both dense and locally connected in $\overline X$ ie for all open connected subset $\W\subset \overline X$ then $\iota(X)\cap \W$ is connected. Fox proved that every spread admits a unique completion up to spread isomorphism, moreover its fonctoriality can then be expressed as follows:
if $X_i\xrightarrow{p} Y_i$ are spreads $i\in \{1,2\}$ and $(f,g)$ with $ f:Y_1\rightarrow Y_2, g:X_1\rightarrow X_2$ is a morphism of spread then there exists a unique map $h:\overline X_1\rightarrow \overline X_2$ lifting $g$ with $\overline X_i$ the completion of $X_i$ for $i\in \{1,2\}$. Since we apply spread results to $(G,X)$-manifold, we work in the context of first countable Hausdorff locally path connected topological spaces and one can check that, given a spread $X\rightarrow Y$, if $Y$ is first countable (resp. Hausdorff) then $X$ is first countable (resp. Hausdorff). Furthermore, if $X$ and $Y$ are second countable, then the completion $\overline X$ is also second countable.

Our starting point is then a notion of topological {\it branched covering} of a connected locally path connected Hausdorff first countable topological space $M$ as a connected surjective complete spread, the ordinary locii of which are open dense and locally connected both in the total space and the base space.This notion has already been considered at least implicitely by Fox \cite{MR0123298} and we refer to Montesinos \cite{MR3060537} for an up to date review.
The first contribution of the present work is a the construction of  a universal covering branched over a given locus as well as its path description akin to the description of the universal covering as a space of homotopy classes of path. We continue furnishing the state of the art by giving more precise descriptions of the fiber above a given point and by stating a Galois correspondance for topological branched covering. Such a correspondance is not new although to the author's knowledge it has not been expressed in the generality considered here which allows for instance Galoisian covering of the 3-sphere branched over a wild cantor. Several difficulties arise: the path description make use of a almost homotopic equivalence relation and a branched covering need not be Galoisian even if its ordinary part is Galoisian. A caracterisation of Galoisian cover among  

Since to the author knowledge no litterature describe general properties of abstract singular $(G,X)$-manifold, we provide elements of a theory of singular $(G,X)$-manifolds. A singular $(G,X)$-manifold is a secound countable Hausdorff topological space $M$ endowed with a $(G,X)$-structure almost everywhere ie on a dense open locally connected in $M$. Two key properties of such manifolds are given. First, the $(G,X)$-structure if extendable, can only be in a unique manner; therefore, there is a maximal regular locus which complement is an irreducible singular locus. Second, a.e. $(G,X)$-morphism are defined as continuous maps whose (co)-restriction to some open dense and locally connected subset is a $(G,X)$-morphism; such maps also have a maximal regular locus and an irreducible singular locus; furthermore, they preserver regular locii wherever they are local homeomorphisms.

Branched coverings of connected singular $(G,X)$-manifolds $M$ are then defined.
The main instance of such a branched covering we shall consider is the spread completion of the universal covering of the regular locus of $M$: the {\it universal covering branched over $\sing(M)$}.
We note that in this instance, the ordinary locus of the spread and the regular locus of the singular $(G,X)$-structure are identical.
To confirm this construction is appropriate to our purpose, we need properties akin to Galoisian covering, branched covering of riemannian surfaces and $(G,X)$-manifolds. To begin with, the deck transformation group, ie the Galois group of $\widetilde M \rightarrow M$, acts via a.e. $(G,X)$-morphisms and we can reuse the preliminary work on topological branched covering to have conditions ensuring that the Galois group acts transitively on the fibers and that $M$ is indeed the quotient of $\widetilde M$ by its Galois group . 
The next objective is to construct the developping map of $M$ by extending by continuity the developping map of $\reg(M)$ provided such an extension exists. We follow a route through properties of the local models of the base space. A singular $(G,X)$-manifold $M$ is locally modeled on $(X_\alpha)_{\alpha\in A}$ with each $X_\alpha$ being singular $(G,X)$-manifolds if each singular point of $M$ admits an open neighborhood isomorphic to an open neighborhood of a singular point of some $X_\alpha$. Provided it exists, under suitable assumptions, the branched covering of the local models can be related to local models of the universal covering branched over the singular locus. Furthermore, the existence of a developping map of $M$ is ensured by the existence of a developping map for the local models. 

As a simple example of the procedure, consider again the example of a finite volume complete hyperbolic surface $\Sigma$. The ideal compactification $\overline \Sigma$ of $\Sigma$ then admits a universal covering branched over the ideal points which can be identified as the the union of $\H^2$ and the set of points in $\partial \H$, the boundary of $\H$, fixed by a parabolic isometry in the holonomy of $\Sigma$. A basis of neighborhoods of a point of $\partial \H$ is given by the horodisks. Notice that the developping map $\widetilde \Sigma \rightarrow \H$ does not extends continuously to the universal covering of $\overline \Sigma$ branched over the ideal point: the image of a singular point should be sent to a point in $\partial \H$. One can then consider $\overline \H:=\H\cup\partial \H$ instead as a local model for $\Sigma$ and $\overline \Sigma$. This does not change the singular locus, but then universal branched covering of the local model of an ideal point is a horodisk (with its ideal point). In particular, the developping map extends injectively for local models and then to the universal branched covering of $\overline \Sigma$. Though the example of finite volume complete hyperbolic surfaces is elementary, we apply the present work to $\mass{0}$-manifolds which are 3-dimensional Lorentzian analogue of compact hyperbolic surface with cusps.
\subsection*{Acknowledgements}
This work is part of a project that has received funding from the European Research Council (ERC) under the European Union's Horizon 2020 research and innovation programme (grant agreement ERC advanced grant 740021--ARTHUS, PI: Thomas Buchert). The author thanks Thierry Barbot for his careful reading and corrections of the manuscipt, José María Montesinos Amilibia for his interest and advices,  Thomas Buchert for his continuous support,  Masoud Hasani,  Uira Noberto Matos de Almeida,  Ivan Izmestiev and Roman Prosanov for valuable discussions.

\section{Branched Covering à la  Fox from spreads to Galois correspondance}
Unless explicitely stated otherwise, all topological spaces are Hausdorff, locally path connected and first countable.
\subsection{Preliminaries on spreads}

\begin{defi} A spread $X\xrightarrow{p} Y$ is a continuous map such that the connected components of the preimages of open subsets of $Y$ generates the topology of $X$.

We define the category of spreads $\CatSpread$ whose objects are spreads and morphisms in $\mathrm{Hom}(X\xrightarrow{p} Y, X'\xrightarrow{p'} Y')$ are couples of continuous functions $(f,g)$  such that the following diagram commutes

$$\xymatrix{
X\ar[d]^p \ar[r]^g & X' \ar[d]^{p'} \\ Y \ar[r]^{f} & Y'
}.$$

\end{defi}

\begin{rem} Let $X\xrightarrow{p} Y$ be a spread, 
with $\U,\V$ open of $X$  and $Y$ respectively such that $p(\U)\subset \V$, the (co)restriction $p_{|\U}^{|\V}$ is a spread.
\end{rem}
\begin{defi}
 Let $X\xrightarrow{p} Y$ be a spread, the {\it ordinary locus} of $Y$  is the set of point $y\in Y$  for which there exists an open neighborhood $\U$ evenly coverd by $p$ ie such that $p$ maps each connected component of $p^{-1}(\U)$ homeomorphically onto $\U$. The ordinary locus of $X$ is the inverse image of the ordinary locus of $Y$.
\end{defi}
\begin{rem}
 The ordinary locii are open but may not be connected. The only obstruction for the restriction $\ord(X)\rightarrow \ord(Y)$ to be a covering is the connectedness of $\ord(X)$.
\end{rem}

\begin{lem}\label{lem:ord_iso} Let $X_i\xrightarrow{p_i} Y_i$ be spreads for $i\in \{1,2\}$ and let $(f,g)$ be a spread isomorphism from $X_1\rightarrow Y_1$ to $X_2\rightarrow Y_2$. Then $g(\ord(X_1)) = \ord(X_2)$ and $f(\ord(Y_1)) = \ord(Y_2)$.
 
\end{lem}
\begin{proof}
 Let $\U$ be an open connected neighborhood of $p_1(x)$ evenly covered by $p_1$, let $(\hat\U_i)_{i\in I}$ the connected components of $p_1^{-1}(\U)$ and define $\V:=f(\U)$ and $\hat \V_i:= g(\hat \U_i)$ for $i\in I$. 
 Since $f$ and $g$ are bijective and $f\circ p_1=p_2\circ g$, we have $$p_2^{-1}(\V)=g\circ p_1^{-1}\circ f^{-1}(\V) = g\circ p_1^{-1}(\U) = \bigcup_{i\in I} \hat \V_i.$$   
 
 Furthermore, $p_{2|\hat\V_i}^{|\V}\circ g^{|\hat\V_i}_{|\hat\U_i} = f_{|\U}^{|\V}\circ p_{1|
 \hat\U_i}^{|\U}$ for all $i\in I$ and $g^{|\hat\V_i}_{|\hat\U_i}$ as well as $f_{|\U}^{|\V}\circ p_{1|
 \hat\U_i}^{|\U}$ are homeomosphims; therefore, for all $i\in I$, $p_{2|\hat\V_i}^{|\V}$ is an homeomorphism.
 Finally, $\V$ is evenly covered by $p_2$. The result follows.
 \end{proof}

\begin{defi}[Portly/skeletal subset] Let $X$ be a Hausdorff locally connected topological space, a subset $\U$ of $X$ is portly (in $X$) if $\U$ is open dense and locally connected in $X$. A subset $S\subset X$ is skeletal if its complement in $X$ is portly.
\end{defi}

\begin{example}
 The $k$-skeleton of a pure $n$-dimensional simplicial complex $X$ is skeletal for $k\leq n-2$.
 \end{example}

Portly subsets have nice properties, the proof of which rely on usual connectedness and/or density arguments and are thus left to the reader. Consider a locally connected Hausdorff topological space $X$.
\begin{itemize}
 \item finite intersections of portly subsets are portly;
 \item if $\U\subset \V$, $\U$ is portly in $X$ and $\V$ is open then $\V$ is portly (hence unions of portly subsets are portly);
 \item if $\U\subset \V$ and $\U$ is portly in $X$ then $\U$ is portly in $\V$;
 \item if $\U$ is portly in $\V$ and $\V$ is portly in $X$ then $\U$ is portly in $X$;

 \item locally portly subsets are portly (i.e. if $(\U_i)_{i\in I}$ is a familly of open subsets and for each $i\in I$, $\V_i\subset \U_i$ is a portly subset of $\U_i$, then $\bigcup_{i\in I}\V_i$ is portly in $\bigcup_{i\in I}\U_i$);
 \item if $X\xrightarrow{f} Y$ is a local homeomorphism and $\U\subset Y$ is portly in $Y$ then $f^{-1}(\U)$ is portly in $X$.
\end{itemize}

\begin{defi} A subset $D$ of $X$ is locally connected in $X$ if for all open connected subset $\U$ of $X$, the intersection $\U\cap D$ is connected.
\end{defi}

\begin{defi} Let  $X\xrightarrow{p} Y$ be a spread. For $y\in \overline{p(X)}$, consider the set $X_y$ of maps $\chi : \{\U\subset Y ~ \text{ open and connected neighborhood of }y\} \rightarrow \{\V\subset X ~ \text{open connected}\}$ increasing for the inclusion such that $\chi(\U)$ is a connected component of $p^{-1}(\U)$.
The spread $X\xrightarrow{p} Y$ is complete if for all $y\in \overline{p(X)}$ and all $\chi \in X_y$, the intersection of $\chi(\U)$ for $\U$ going through all open connected neighborhood of $y$ is non empty.

A completion of a spread $X\xrightarrow{p} Y$ is a spread $X'\xrightarrow{p'} Y$ together with an injective spread morphism $\iota$ lifting the identity on $Y$
$$\xymatrix{
X\ar[d]^p \ar[r]^\iota & X' \ar[d]^{p'} \\ Y \ar@{=}[r] & Y
} $$
such that $X'$ is complete and the image of $\iota$ is open dense and locally connected in $X'$.
\end{defi}
Recall that a functor $T$ is fully faithful if the induced function $\Hom(X,Y) \rightarrow \Hom(T(X),T(Y))$ is bijective.
\begin{thm}[\cite{MR0123298,MR2198590}] \label{theo:Fox_completion_functor}
 Every spread admits a unique completion up to isomorphism, furthermore the completion of spread is a fully faithful functor  from $\CatSpread$ to itself.
\end{thm}

\begin{prop}[\cite{MR2198590}] \label{prop:spread_fiber}
 Let $X\xrightarrow{p} Y$ be a complete spread with $Y$ Hausdorff and let $b\in Y$. Let $\boldsymbol{\U}$ be a connected neighborhood basis of $b$. For $\U \in \boldsymbol{\U}$, let $X_\U$ be the space of connected components of $p^{-1}(\U)$ endowed with the discrete
 topology. Define bonding maps $X_\U\xrightarrow{\pi_{\U\V}} X_\V$ by setting for $\U\subset \V$ the image $\pi_{\U\V}(x)$ of $x\in X_\U$ to be the connected component of $x$ in $X_\V$.
 Then, $$p^{-1}(b) \simeq \varprojlim_{\U\in \boldsymbol{\U}}X_\U$$
 the projective limit being taken in the category of Hausdorff topological spaces.
\end{prop}
\begin{proof}
 Consider the map 
 $p^{-1}(b) \xrightarrow{\phi}\varprojlim X_\U$ defined by $\phi(x) = \left(x_\U\right)_{\U\in \boldsymbol{\U}}$ with $x_\U$ the connected component of $x$ in $p^{-1}(\U)$. For every $x\in p^{-1}(b)$, since $p$ is a spread, a neighborhood basis of $x$ is given by the connected components of $x$ in each $p^{-1}(\U)$ for $\U\in \boldsymbol{\U}$. If two elements $x_1,x_2 \in X$ have the same image by $\phi$ then $x_1,x_2 \in \phi(x_1)_\U$ for all $\U\in \boldsymbol{\U}$. Hence, $x_1,x_2$ have a common neighborhood basis. The space $Y$ is Hausdorff, then so is $X$ and then $x_1=x_2$. Therefore, $\phi$ is injective.
 Let $\hat{\boldsymbol{\U}}=\left(\hat\U\right)_{\U\in \boldsymbol{\U}} \in \varprojlim_{i\in I}X_i$, since $X\xrightarrow{p} Y$ is complete, the intersection $\bigcap_{\U\in \boldsymbol{\U}}\hat\U$ is non empty, hence $\phi$ is surjective; thus bijective. Furthermore, $\phi^{-1}(\hat{\boldsymbol{\U}}) = \bigcap_{\hat\U\in \hat{\boldsymbol{\U}}}\hat\U$. 
 
 Notice that for all $\U\in \boldsymbol{\U}$ and all $\V\in X_\U$, $$\phi(\V\cap p^{-1}(b)) = \pi_\U^{-1}(\V)$$
 {with $\pi_\U$ the natural map $\varprojlim_{\W\in \boldsymbol{\U}} X_\W \xrightarrow{~~~} X_\U$.}
 On the one hand, such $\pi_\U^{-1}(\V)$ form a basis of the topology of $\varprojlim X_\U$. On the other hand, since $p$ is a spread, such $\V\cap p^{-1}(\V)$ form a basis for the topology of $p^{-1}(b)$. 
 Finally, $\phi$ is an homeomorphism.
 
\end{proof}
\begin{cor}[\cite{MR2198590}]\label{cor:totdiscont} Let $X\xrightarrow{p} Y$ be a spread, then the fibers of $p$ are totally discontinuous. 
\end{cor}

\begin{defi}[Automorphisms of a spread] Let $X\xrightarrow{p} Y$ be a spread, an automorphisms of $p$ is an homeomorphism $\varphi$ of $X$ such that $p\circ \varphi = p$. 
Denote by  $\Gamma(X/Y)$ the group of spread automorphisms of $X$ above $Y$. 
We also denote by $\Gamma_x$ the stabilizer of $x$ if $x\in X$ and by $\Gamma_\U$ the set-wise stabilizer of $\U$ if $\U$  is a subset of $X$.
\end{defi}
Note that a spread automorphism is just an automorphism in the category $\CatSpread$.

\begin{lem}\label{lem:group_action} Let $X\xrightarrow{p}Y$  be a spread and let $\Gamma$ be a  totally discontinuous topological group acting continuously on $X$. 

If $\Gamma$ acts by spread automorphisms, then the action of $\Gamma$ extends uniquely to a continuous action by automorphisms on the completion of $X\xrightarrow{p}Y$.
\end{lem}
\begin{proof} 
The spread morphism 
$$\xymatrix{\Gamma\times X \ar[rr]^{(g,x)\mapsto g\cdot x} \ar[d]^p && X \ar[d]^p\\Y \ar@{=}[rr] && Y} $$
 extends naturally to the spread completions. Since $\Gamma$ is totally disconnected, connected open sets in $\Gamma\times X$ lie in a sheet $\{g\}\times X$, hence the completion of $\Gamma\times X$ is $\Gamma \times \overline X$ with $\overline X$ the completion of $X$. Finally,  the natural map $\Gamma\times \overline X \rightarrow \overline X$ is continuous and defines a continuous group action of $\Gamma$ on $\overline X$.
\end{proof}

\subsection{Primer on Branched covering} 
As before, unless explicitely stated otherwise, all topological spaces are Hausdorff, locally path connected and first countable. In the present section as well as the following, we will make intensive use of the following Lemma which, for brievety sake, we will not refer to systematically.
\begin{lem}\label{lem:path_covering} Let $X$ be a connected first countable, locally path connected, Hausdorff topological space and let $\U\subset X$ be a portly subset. 
For all $a\in \U$ and all $b\in X$, there exists a path $\gamma: [0,1]\rightarrow X$  from $a$ to $b$ such that $\gamma([0,1[)\subset \U$.
\end{lem}
\begin{proof}
 Let $(a,b) \in \U\times X$ and let $(\V_n)_{n\in \N}$ be a decreasing base of open connected neighborhood of $b$. Since $\U$ is portly, in particular $\U$ is dense and there exists a sequence $(b_n)_{n\in \N}$ such that  $\forall  n\in \N, b_n \in \V_n\cap \U$. Since $\U$ is portly and $X$ is connected then $\U$ is connected and we can choose a path $\gamma_0$ in $\U$ from $a$ to $b_0$. Then for $n\in \N$, $b_n$ and $b_{n+1}$ are both in $\U\cap \V_n$ and since $\U$ is portly and $\V_n$ open connected, $\U\cap \V_n$ is open connected (hence path connected) and we can choose a path $\gamma_{n+1}$ in $\U\cap \V_n$ from $b_n$ to $b_{n+1}$. The concatenation of the sequence of path $(\gamma_n)_{n\in \N}$ is a path $\gamma:[0,1[\rightarrow \U$ such that $\lim_{1^-}\gamma = b$. The result follows.
\end{proof}

\begin{defi}[Branched covering]  A {\it branched covering} is a complete surjective spread $X\rightarrow Y$, with $X$ and $Y$ connected, whose ordinary locii in $X$ and $Y$ are portly subsets of $X$ and $Y$ respectively. We call $Y$ the {\it base space}, $X$ the {\it total space} and the inverse image of a point $a$ of $Y$ is called the {\it fiber} above $a$.
 
 We say that $X\rightarrow Y$ is branched over $Y\setminus \ord(Y)$. The latter is the branching locus of $X\rightarrow Y$ and for $S\subset Y$, we say that $X\rightarrow Y$ is possibly branched over $S$ if $S$ contains his the branching locus.

 \end{defi}
 
 \begin{rem} The branching locii of a branched covering are skeletal thus closed with empty interior.
 \end{rem}

 We begin by a branched covering version of the lifting property for complete spreads ("Corollary of the extension Theorem" in Fox \cite{MR0123298}).
\begin{lem}\label{lem:complete_path_carac} 
 Let $X\xrightarrow{p} Y$ be a spread which induces a covering $X_1\rightarrow Y_1$ for some $X_1$ and $Y_1$ portly subsets of $X$ and $Y$ respectively. Then the following are equivalent:
 \begin{enumerate}[(i)]
  \item $X\rightarrow Y$ is a branched covering;
  \item $X\rightarrow Y$ is a complete spread
  \item for every Hausdorff locally connected topological space $Q$ and every homotopy $H:Q\times [0,1]\rightarrow Y$, every partial lift 
  $\hat H : Q \times [0,1[ \rightarrow X_1$ of $H$ extends continuously to some lift $Q \times [0,1] \rightarrow X$ of $H$;
  \item for every path $\gamma:[0,1]\rightarrow Y$ every partial lift 
  $\hat \gamma : [0,1[ \rightarrow X_1$ extends continuously to some lift $[0,1] \rightarrow X$ of $\gamma$.
  
 \end{enumerate}
\end{lem}
\begin{proof}
 One has $(i)\Rightarrow (ii)$ by definition, $(ii)\Rightarrow (iii)$ is a direct consequence of the corollary of Fox mentionned above and $(iii)\Rightarrow (iv)$ is trivial (it suffices to take $Q$ a singleton). 
 Assume $(iv)$ and consider some $y_0 \in Y_1$. Since $Y_1$ is portly in $Y$, for every $y\in Y$ by Lemma \ref{lem:path_covering} there exists a path $\gamma:y_0\rightsquigarrow y$ such that $\gamma([0,1[)\subset Y_1$. Since $X_1\rightarrow Y_1$ is a covering, $\gamma_{|[0,1[}$ lifts to some $\hat\gamma :[0,1[\rightarrow X_1$. By assumption, $\hat\gamma$ admits a limit $x$ at $1$ and by continuity, $p(x)=y$. Therefore, $p$ is surjective.  
 Consider $\overline X \rightarrow Y$ the completion of $X\rightarrow Y$; without loss of generality, we may assume that $X\subset \overline X$ and also denote by $p$ the map $\overline X\rightarrow Y$.   Notice that since $X_1$ is portly in $X$ and $X$ is portly in $\overline X$ then $X_1$ is portly in $\overline X$. Let $x \in \overline X$ and let $\hat y_0$ be some lift of $y_0$ in $X_1$. Let  $\hat \gamma$ be some path from $\hat y_0$ to $x$ in $\overline X$ such that $\hat\gamma([0,1[)\subset X_1$. Since $\hat\gamma_{|[0,1[}$ is a lift in $X_1$ of $p\circ \hat \gamma$ is a path in $Y$, by assumption, $\hat\gamma_{|[0,1[}$ has limit at $1$ in $X$. Therefore $x = \hat\gamma(1)= \lim_{1^-} \hat\gamma_{|[0,1[} \in X$. Finally, $\overline X\subset X$ hence $\overline X = X$ and $X\rightarrow Y$ is complete. 
\end{proof}
\begin{rem} 
 In Lemma \ref{lem:complete_path_carac}, one can take $X_1 = \ord(X)$ and $Y_1=\ord(Y)$.
\end{rem}

\begin{lem}\label{lem:spread_connected_image}
   Let $X\xrightarrow{p} Y$ be a branched covering, then for every connected open $\U\subset Y$ and every connected component $\V$ of $p^{-1}(\U)$ we have $p(\V)=\U$. In particular, $p$  is open.
\end{lem}
\begin{proof}
 To begin with, the (co-)restriction $p_{|\ord(X)}^{|\ord(Y)}$ of $p$ to the regular locus is a covering; indeed, $\ord(X)$ is connected. Let $\U\subset Y$ open and connected and let $\V$ be a connected component of $p^{-1}(\U)$.

 \begin{itemize}
  \item Since $\ord(X)$ is dense, we can choose some $\hat x\in \ord(X)\cap \V$; define $x:=p(\hat x)$. Since $Y$ is locally pathwise connected and since $\ord(Y)$ is open, dense and locally connected in $Y$, then $\ord(\U):=\ord(Y)\cap \U$ is path connected.  We then can and do consider a path $\gamma:[0,1]\rightarrow \ord(Y)$ from $x$ to some $y\in \ord(Y)$. Consider the unique lift $\hat\gamma$ of $\gamma$ in $\ord(X)$ such that $\hat\gamma(0)=\hat x$. On the one hand, $\hat\gamma([0,1])\subset p^{-1}(\U)$; on the other end, $\hat\gamma([0,1])$ is connected. Therefore, all points of $\hat \gamma$ are in $\V$ and $\gamma([0,1])\subset p(\V)$. 
  
  The path $\gamma$ is arbitrary, so $p(\ord(\V))=\ord(\U)$.
  \item Let $(x,y)$ in $\ord(\U)\times\U$, let $\hat x\in p^{-1}(x)\cap \V$ and let $\gamma:[0,1[\rightarrow \ord(\U)$ a path such that $\gamma(0)=x$ and $\lim_{1^-}\gamma = y$. Let $\hat\gamma$ the unique lift of $\gamma$ to $\ord(X)$ such that $\hat\gamma(0)=\hat x$.  For any open neighborhood $\W$ of $y$, consider $t_\W$ the smallest $t\in [0,1[$ such that $\gamma(]t_\W,1[)\subset \W $; then define $\hat \W$ the unique connected component of $p^{-1}(\W)$ which contains $\hat\gamma(]t_\W,1[)$. The map $\W\mapsto \hat\W$ is increasing for the inclusion, since $X\rightarrow Y$  is a complete spread, $\cap_{\W} \hat\W$ is a singleton say $\{\hat y\}$. In particular, $\lim_{1^-}\hat\gamma = \hat y$ and $p(\hat y)=y$. 
  
  Finally, $y\in p(\V)$.

  \end{itemize}
  
\end{proof}
\begin{lem}[\cite{MR2198590}]\label{lem:branched_covering_restriction}
 Let $X\xrightarrow{p} Y$ be a branched covering, let $\U$ be some open connected subset of $Y$ and let $\V$ be a connected component of $p^{-1}(\U)$. Then $\V\xrightarrow{p} \U$ is a branched covering.
\end{lem}
\begin{proof}
For clarity sake, we denote by $\V\xrightarrow{q} \U$ the restriction of $p$ to $\V$.
  The map $\V\xrightarrow{q} \U$ is a spread and by Lemma \ref{lem:spread_connected_image}, $q$ is surjective.  
  
  Let $y\in \U$ and let $\W$ be an open connected neighborhood of $y$ in $\U$ evenly covered by $p$. Let $\hat\W$ be a connected component of $p^{-1}(\W)$, either $\hat\W$ and  $\V$ are disjoint or 
  $\hat\W\cap \V\neq \emptyset$ in which case $\hat\W\cup\V$ is connected and in $p^{-1}(\U)$ thus $\hat\W\subset \V$. Hence, the connected components of $q^{-1}(\W)$ are exactly the connected components of $p^{-1}(\W)$ that are in $\V$ and $\W$ is evenly covered. 
  The ordinary locus $\ord(q)$ of $p_{|\V}^{|\U}$ is then the intersection of the ordinary locus $\ord(p)$ of $p$ with $\V$. Furthermore, $\ord(p)$ is a dense, locally connected in $X$ open subset thus  its intersection with $\V$ is open,  dense and locally connected in $\V$. 
  
   Let $(\U_i)_{i\in I}$ be  the set of open connected neighborhoods of some $y\in Y$ and let $(\W_i)_{i\in I}$ be a familly of connected open subsets of $X$ such that $\W_i$ is a connected component of $q^{-1}(\U_i)$ for all $i\in I$, and such that $\W_i\subset \W_j$ if $\U_i\subset \U_j$ for all $i,j\in I$. Consider $\hat\W_i$ the connected component of $\W_i$ in $p^{-1}(\U_i)$. On the one hand, $p$ is a complete spread, there exists some $x\in X$ such that $\cap_{i\in I} \hat\W_i=\{x\}$. On the other hand, the same way as before, the connected components of $p^{-1}(\W_i)$ are either in $\V$ or disjoint from $\V$; therefore,  $\W_i= \hat\W_i$.  In particular, $x\in \V$ and $\cap_{i\in I}\W_i =\{x\}$. Finally,  $\V\xrightarrow{q} \U$ is complete thus a branched covering.

\end{proof}
\begin{cor} 
  Let $X\xrightarrow{p} Y$ be a branched covering, let $\U$ be some portly subset of $Y$ then $p^{-1}(\U)\xrightarrow{p}\U$ is a branched covering.
\end{cor}
\begin{proof} In view of Lemma \ref{lem:spread_connected_image}, it suffices to prove $p^{-1}(\U)$ is connected and since $X$ is connected is suffices to prove $p^{-1}(\U)$ is portly. Since $\U$ and $\ord(Y)$ are portly in $Y$, then $\U\cap \ord(Y)$ is portly in $Y$ hence in $\ord(Y)$. The map $\ord(X)\rightarrow \ord(Y)$ is a covering hence a local homeomorphism and $p^{-1}(\U\cap \ord(Y))$ is then portly in $p^{-1}(\ord(Y))=\ord(X)$. Furthermore, $\ord(X)$ is portly in $X$ then $ p^{-1}(\U\cap \ord(Y))$ is portly in $X$ and since $p^{-1}(\U\cap \ord(Y))$ is a subset of $p^{-1}(\U)$ which is open, $p^{-1}(\U)$ is portly in $X$.
 Finally, $X$ is connected then so is $p^{-1}(\U)$
\end{proof}

\begin{rem} Let $X\rightarrow Y$ be a branched covering, the spread induced on the ordinary locii $\ord(X)\rightarrow \ord(Y)$ is a covering.
\end{rem}
\begin{rem} The composition of branched coverings is a branched covering.
\end{rem}

\begin{lem}\label{rem:completion_covering} Let $X\xrightarrow{p} Y$ be a branched covering and $\U\subset Y$ some portly subset. The spread $X\rightarrow Y$ is isomorphic to the completion of $p^{-1}(\U)\rightarrow Y$.
\end{lem}
\begin{proof}
 $X\rightarrow Y$ is a completion of the spread $p^{-1}(\U)\rightarrow Y$ which is unique up to isomorphism by Theorem~\ref{theo:Fox_completion_functor}.
\end{proof}

Recall that a connected, locally path connected Hausdorff topological space $X$ is {\it semi-locally simply connected} if every $x\in X$ admit a neighborhood $\U$ such that every loop $\gamma$ in $X$ of base point $x$ is trivial in $X$. By standard results \cite{MR1700700}, it a necessary and sufficient condition for such a topological space $X$ to admit a universal covering.

\begin{prop}\label{prop:universality} Let $Y$ be a connected, locally path connected, Hausdorff topological space.  Let $S$ be a skeletal subset of $Y$.

If $Y\setminus S$ is semi-locally simply connected then  there exists a covering $\widetilde Y^S$ of $Y$ possibly branched over $S$ which is maximal among such branched covering.
Furthermore, $\widetilde Y^S$
 is unique up to isomorphism and universal in the sense that for any covering $X\xrightarrow{p} Y$ possibly branched over $S$, there exists a covering $\widetilde Y^S\rightarrow X$ possibly branched over $p^{-1}(S)$.
\end{prop}
\begin{rem}
One could sum up the proof below by saying Fox completion induces a category equivalence between the category of covering of $Y\setminus S$  and the category of covering of $Y$ possibly branched over $S$.
\end{rem}
\begin{proof}

 Assume $Y\setminus S$ is semi-locally simply connected. Then define $\widetilde Y^S$ the completion of spread $\widetilde {Y\setminus S} \rightarrow Y$ with $\widetilde {Y\setminus S}$ the universal covering of $Y\setminus S$. For any branched covering $X\xrightarrow{p}Y$ possibly branched over $S$, since $Y\setminus S\subset \ord(Y)$ is portly, $p^{-1}(Y\setminus S)$ is portly (hence connected) and $p^{-1}(Y\setminus S) \rightarrow Y\setminus S$ is a covering. Hence, there exists a covering morphism $\widetilde {Y\setminus S} \rightarrow p^{-1}(Y\setminus S) $. Since the completions of 
 $p^{-1}(Y\setminus S) \rightarrow Y$ and $\widetilde {Y\setminus S}\rightarrow Y$ are respectively $X$ and $\widetilde Y^S$, by Theorem \ref{theo:Fox_completion_functor} we obtain a morphism of branched covering 
 $\widetilde Y^S \rightarrow X$ which is itself a branched covering and its ordinary locus in $X$ contains $p^{-1}(Y\setminus S)$. Finally, $\widetilde Y^S \rightarrow X$ is possibly branched over $p^{-1}(S)$. 
 
 If $X\rightarrow Y$ is a maximal  covering possibly branched over $S$ then the induced covering $p^{-1}(Y\setminus S)\rightarrow Y\setminus S$  is maximal among covering of $Y\setminus S$ hence its universal covering. Therefore, $X$ is isomophic to $\widetilde Y^S$.
\end{proof}
\subsection{Paths and universal branched covering}
The abstract description of Fox completion of a spread, though very general, may benefit from a path point of view. This section is devoted to a path description of the universal cover possibly branched above some locus akin to  description as set of homotopy classes of path of the universal cover of a Hausdorff, connected, locally path connected, semi-locally simply connected topological space. 

To begin with, we introduce some notations as well as a refinement of the homotopy equivalence relation that can keep track of branched locus of a branched covering.
Considering some connected Hausdorff, locally path connected topological space $X$, some connected open subset $\U$ and some $(x,y)\in \U\times \overline\U$ we denote by $\Omega(x,y,\U)$ the set of path $\gamma:[0,1]\rightarrow \U$ such that $\gamma(0)=x$,$\gamma(1)=y$ and $\gamma(]0,1[)\subset \U$. This set will always be endowed with its compact-open topology and two paths $\gamma_1,\gamma_2$ of $\Omega(x,y,\U)$ are homotopic with respect to $\U$ if they are in the same connected component of $\Omega(x,y,\U)$ or, equivalently, if there exists an homotopy $H:[0,1]^2\rightarrow X$ with fixed ends such that $H(0,\cdot)=\gamma_1$, $H(1,\cdot)=\gamma_2$ and such that $H([0,1]\times ]0,1[)\subset \U$. The homotopic relation is not suitable for our purpose as this section as well as the following will make clear, we thus introduce a weaker equivalence relation.

\begin{defi}[Almost trivial loop] Let $X$ be a Hausdorff, locally path connected topological space, let $\U\subset X$ be an open subset and let $x\in \overline \U$.

A loop $\gamma\in \Omega(x,x,\U)$ is  almost trivial with respect to $\U$ if its connected component in $\Omega(x,x,\U)$ intersects $\Omega(x,x,\V\cap \U)$ for all neighborhood $\V$ of $x$. 
\end{defi}
\begin{example}
 Consider $Y = \RR^3$, $S_1= \RR\times\{0\}^2$ and $$S_2=\{(x,0,\sin(1/x)) ~:~ x\in \RR^*\}\cup \{(0,0)\}\times [-1,1];$$ denote $O=(0,0,0)$ the origin. The fundamental group of $Y\setminus S_i$ is isomorphic to $\ZZ$ for both $i=1$ and $i=2$ and notice that all loops of $\Omega(O,O,Y\setminus S_i)$ are almost trivial. 
 However, while a loop doing several turn around $S_1$ can be contracted to $O$, such a loop around $S_2$ cannot.  Hence, almost  trivial is indeed different from homotopically trivial.
\end{example}
\begin{defi}[Almost homotopic paths]
 Let $X$ be a Hausdorff, locally path connected topological space, let $\U\subset X$ be an open subset and let $(x,y)\in \U\times \overline \U$.

Two paths $\gamma_1,\gamma_2 \in \Omega(x,y,\U)$ are almost homotopic with respect to $\U$ if the loop $\gamma_2^{-1}*\gamma_1$ is almost trivial with respect to $\U$.
\end{defi}

\begin{rem} Implicitely, two almost homotopic paths $\gamma_1,\gamma_2$ always satisfy $\gamma_1(1)=\gamma_2(1)$ and $\gamma_1(0)=\gamma_2(0)$. 
\end{rem}

\begin{lem} \label{lem:almost_homotopic_carac}
 Let $X$ be a Hausdorff, first countable, connected, locally path connected topological space, let $\U\subset X$ be an open subset and let $x \in  \overline \U$.
 
 For $\gamma \in \Omega(x,x,\U)$,  the following are equivalent:
 \begin{enumerate}[(i)]
  \item $\gamma$ is almost trivial with respect to $\U$;
  \item for all open neighborhood $\V$ of $x$ and for all $t_0,t_1\in ]0,1[$ such that $t_0<t_1$ and $\gamma([0,t_0]\cup [t_1,1])\subset \V$, there exists a path $\omega : \gamma(t_0)\rightsquigarrow \gamma(t_1)$ in $\V\cap \U$ such that $\gamma$ is homotopic to $\gamma_{|[0,t_0]}*\omega*\gamma_{|[t_1,t_0]}$ in $\Omega(x,x,\U)$.
 \end{enumerate}
\end{lem}
\begin{proof}
  $(ii)$ trivially implies $(i)$, we thus assume that $\gamma$ is almost trivial with respect to $\U$. Let $\V$ be an open neighborhood of $x$ and let $t_0<t_1$ such that $\gamma([0,t_0]\cup[t_1,1])\subset \V$. Let $H:[0,1]\times [0,1]\rightarrow Y$ be an homotopy from $\gamma$ to some path $\gamma_0\in \Omega(x,x
  ,\V)$ such that $\forall s\in [0,1], H(s,\cdot)\in \Omega(x,x,\U)$. 
  Consider the following closed connected subset $B$ of the boundary of $[0,1]\times [0,1]$:$$B := \{0\}\times [0,t_0]\ \cup\ [0,1]\times \{0\}\ \cup\ \{1\}\times [0,1]\ \cup\ [0,1]\times\{1\}\ \cup\ \{0\}\times [t_1,1].$$
  Since $H(B)\subset \V$, there exists a connected open neighborhood $\W$ of $B$ in $[0,1]\times [0,1]$ such that $H(\W)\subset \V$. Take some path $\eta:[t_0,t_1]\rightarrow \W$ such that $\eta(t_0)=(0,t_0)$ and $\eta(t_1)=(0,t_1)$ and define $\omega:=H\circ \eta$. 
  The homotopy $$\fonctiontrois{J}{[0,1]\times [0,1]}{Y}{s\in [0,1],t\in [0,t_0]}{H(0,t)}{s\in [0,1], t\in [t_0,t_1]}{H\left(s\cdot\eta_1(t),(1-s)t+s\cdot\eta_2(t)\right)}{s\in [0,1],t\in [0,t_0]}{H(0,t)}$$
  is such that $J(0,\cdot)=\gamma$ and $J(1,\cdot)=\gamma_{|[0,t_0]}*\omega*\gamma_{|[t_1,t_0]}$, moreover  $\forall s\in~[0,1],~J(s,\cdot)\in \Omega(x,x,\U)$. The result follows.
\end{proof}

\begin{lem} 
 Let $X$ be a Hausdorff, first countable, connected, locally path connected topological space, let $\U\subset X$ be an open subset and let $(x,y)\in \U\times \overline \U$.
 
 Then, the "almost homotopic" relation is an equivalence relation in $\Omega(x,y,\U)$.
 
\end{lem}
\begin{proof}
The relation is clearly reflexive and symmetric, we shall then prove that it is also transitive. Let $\gamma_1, \gamma_2,\gamma_3 \in \Omega(x,y,\U)$ such that $\gamma_1,\gamma_2$ are almost homotopic with respect to $\U$ and such that
$\gamma_2,\gamma_3$ are almost homotopic with respect to $\U$. Let $\V$ be an open neighborhood of $x$ and let $t \in ]0,1[^3$ such that $\forall i\in\{1,2,3\},\ \gamma_i([t_i,1])\subset \V$. By Lemma \ref{lem:almost_homotopic_carac}, there exists a path $\omega_{1}$ (resp. $\omega_2$)  in $\V\cap \U$ from $\gamma_1(t_1)$ to $\gamma_2(t_2)$ (resp. from $\gamma_2(t_2)$ to $\gamma_3(t_3)$) such that $\gamma_1^{-1}*\gamma_2$ is homotopic to $(\gamma_{1|[t_0,1]})^{-1}*\omega_1*\gamma_{2|[t_1,1]}$ (resp. such that $\gamma_2^{-1}*\gamma_1$ is homotopic to $(\gamma_{2|[t_1,1]})^{-1}*\omega_2*\gamma_{3|[t_3,1]}$ ). Then $(\gamma_{1|[0,t_0]})^{-1}*\gamma_{2|[0,t_2]}*\omega_1^{-1}$ is homotopically trivial and the loop
 $\gamma_{1}^{-1}*\gamma_3$ is homotopic to $$(\gamma_{1|[t_0,1]})^{-1}*\left((\gamma_{1|[0,t_0]})^{-1}*\gamma_{2|[0,t_2]}*\omega_1^{-1}\right)*\omega_1*(\gamma_{2|[0,t_2[})^{-1}*\gamma_3$$ which is thus homotopic to $(\gamma_{1|[t_0,1]})^{-1}*\omega_1*(\gamma_{2|[0,t_2[})^{-1}*\gamma_3$. The same way, the loop $(\gamma_{2|[0,t_1]})^{-1}*\gamma_{3|[0,t_3]}*\omega_2^{-1}$ is homotopically trivial and $\gamma_{1}^{-1}*\gamma_3$ is then homotopic to $(\gamma_{1|[t_0,1]})^{-1}*\omega_1*\omega_3*\gamma_{3|[t_3,1]} \in \Omega(x,x,\V\cap \U)$. 
 Finally, $\gamma_{1}^{-1}*\gamma_3$ is almost trivial with respect to $\U$ and $\gamma_1$ is almost homotopic to $\gamma_3$.
\end{proof}

\begin{lem} \label{lem:almost_homotopic_carac2}
 Let $\widetilde Y^S\xrightarrow{p} Y$ be a universal covering possibly ramified above $S$, let $(x,y)\in (Y\setminus S)\times Y$ and let $\hat x \in p^{-1}(x)$.
Let $\gamma_1, \gamma_2\in \Omega(x,y,Y\setminus S)$ and let $\hat \gamma_1,\hat\gamma_2$ be the respective lifts of $\gamma_1$ and $\gamma_2$ such that $\hat \gamma_1(0)=\hat \gamma_2(0)=\hat x$. 

Then, $\hat\gamma_1(1)=\hat\gamma_2(1)$ if and only if $\gamma_1$ and $\gamma_2$ are almost homotopic.
\end{lem}
\begin{proof}
 Assume $\gamma_1$ and $\gamma_2$ are almost homotopic. For  $\U$ some standard neighborhood of $\hat\gamma_1(2)$. By Lemma \ref{lem:spread_connected_image}, $p(\U)$ is a connected open neighborhood of $y$, and since $\gamma_1$ and $\gamma_2$ are almost homotopic, $\gamma_1$ is homotopic to a path $\gamma_1'$ of the form $\gamma_{2|[0,t_2]}*\omega*\gamma_{1|[t_1,1]}$ with $t_1\in ]0,1[$
 such that $\gamma_i([t_i,1])\subset p(\U)$ for $i\in \{1,2\}$ and $\omega\in  \Omega(\gamma_2(t_2),\gamma_1(t_1),Y\setminus S)$. 
 Since $\gamma_1$ and $\gamma_1'$ are homotopic and $p^{-1}(y)$ is totally disconnected, then $\gamma_1(1)=\gamma_1'(1)$. Since $\omega*\gamma_{1|[t_1,1]}\subset p(\U)$, the lift $\eta$ of $\omega*\gamma_{1|[t_1,1]}$ such that $\eta(0)=\gamma_2(t_2)$ stays in $\overline \U$. In particular $\hat\gamma_1(1)=\eta(1)\in \overline \U$. Therefore, $\hat\gamma_1(1)$ is in the closure of all neighborhood of $\hat\gamma_2(1)$, since $\widetilde Y^S$ is Hausdorff we conclude that $\hat\gamma_1(1)=\hat\gamma_2(1)$.
 
 Assume $\hat\gamma_1(1)=\hat\gamma_2(1)$. Let $\U\subset \widetilde Y^S$ be some standard neighborhood of  $\hat\gamma_1(1)$, let $t_1,t_2\in ]0,1[$ such that $\hat\gamma_i([t_i,1])\subset \U$ for $i\in \{1,2\}$ and let $\omega \in \Omega(\hat\gamma_1(t_1),\hat\gamma_2(t_2),\U\cap p^{-1}(Y\setminus S))$. Since $p^{-1}(Y\setminus S)$ is simply connected, the loop $\hat\gamma_{1|[0,t_1]} *\omega * (\hat\gamma_{2|[0,t_2]})^{-1}$ is homotopically trivial. Therefore, $\hat\gamma_2$ is homotopic to $\hat\gamma_{1|[0,t_1]}*\omega*\hat\gamma_{2|[t_2,1]}$, hence $\gamma_2$ is homotopic to $\gamma_{1|[0,t_1]}*(p\circ\omega)*\gamma_{2|[t_2,1]}$. Finally, the loop $(\gamma_1)^{-1}*\gamma_2$ is homotopic to  $(\gamma_1)^{-1}*\gamma_{1|[0,t_1]}*(p\circ\omega)*\gamma_{2|[t_2,1]}$ which is homotopic to $(\gamma_{1|[t_1,1]})^{-1}*(p\circ \omega)*\gamma_{2|[t_2,1]} \subset \U$. 
The loop $(\gamma_1)^{-1}*\gamma_2$ is thus almost trivial and $\gamma_1$ is almost homotopic to $\gamma_2$. 
\end{proof}

\begin{prop} \label{prop:ramified_covering_path} Let $Y$ be connected, locally path connected Hausdorff topological space and let $S\subset Y$ be a skeletal subset. 
Let $y_0\in Y\setminus S$ and let $\Omega:=\Omega(y_0,Y\setminus S)$. Define $X$ as the set of almost homotopy classes of $\Omega$ with respect to $Y\setminus S$ and define the map $\varphi:\Omega\rightarrow Y, \gamma \mapsto \gamma(1)$.

Then, $\varphi$ induces a branched covering $X \xrightarrow{~~\overline\varphi~~ } Y$  isomorphic to $\widetilde Y^S \xrightarrow{~~p~~} Y$.
 
\end{prop}
\begin{proof}

Let $\hat y_0$ be some lift of $y_0$ in $\widetilde Y^S$.
Consider the function $\psi : \Omega\rightarrow \widetilde Y^S$ which associate to $\gamma\in \Omega$ the point $\hat\gamma(1)$ where $\hat\gamma$ is the unique lift of $\gamma$ in $\widetilde Y^S$ such that $\hat\gamma(0)=\hat y_0$.
By Lemma \ref{lem:almost_homotopic_carac2}, $\psi$ induces an injective function $\overline \psi:X\rightarrow \widetilde Y^S$,
the following diagram commutes
$$\xymatrix{X\ar[d]^{\overline \varphi}\ar[r]^{\overline\psi}& \widetilde Y^S \ar[d]^p \\ Y\ar@{=}[r]&Y}$$
and it suffices to prove that $\overline \psi$ is an homeomorphism.

\begin{itemize}
 \item Since $\widetilde Y^S$ is first countable and $p^{-1}(Y\setminus S)$ is portly, for all $x\in \widetilde Y^S$ there exists a path $\hat\gamma\in \Omega(\hat y_0,x,p^{-1}(Y\setminus S))$ and, for such a path,  $\psi(p\circ \hat\gamma)=\hat\gamma(1)=x$. The map $\psi$ is thus surjective, hence bijective.
 
 \item Let $\U$ be a standard neighborhood of some $x\in \widetilde Y^S$ and consider some $\gamma\in \Omega$ such that $\psi(\gamma)=x$. Take some $t_0\in ]0,1[$ such that $\gamma([t_0,1])\subset p(\U)$.
 Since $Y\setminus S$ is semi-locally simply connected, for all $t\in [0,t_0]$ there exists some path connected open neighborhood $\V\subset Y\setminus S$ of $\gamma(t)$ such that all loop in $\V$ are homotopically trivial in $Y\setminus S$. By compactness, we can choose finitely many such neighborhoods $(\V_i)_{i\in\lsem 1,n \rsem}$ as well as an increasing sequence $\alpha\in [0,t_0]^{\lsem 0,n\rsem}$ such that :
 $\alpha_0=0, \alpha_n=t_0$ and $\forall i\in \lsem 1,n\rsem,\ \gamma([\alpha_{i-1},\alpha_i]) \subset \V_i$. 
 Define $\V_{n+1}:=\U$ and $\W_i$ the connected component of $\gamma(\alpha_i)$ in $\V_i\cap\V_{i+1}$ for $i\in \lsem 1,n\rsem$ and define 
 $$\mathcal O := \left\{\eta\in \Omega ~\left|~\begin{array}{l}
                                                     \forall i\in \lsem 1,n\rsem, \eta([\alpha_{i-1},\alpha_i]\subset \V_i)  \\ \forall i\in \lsem 1,n\rsem,\ \eta(\alpha_i)\in\W_i \\ 
                                                     \eta([t_0,1])\subset p(\U)
                                                      \end{array} \right.\right\}.$$
$\mathcal O$ is an open neighborhood of $\gamma$ for the compact-open topology of $\Omega$. Let $\eta\in \mathcal O$, forall $i\in \lsem 1,n\rsem $, choose a path  $\omega_i \in 
\Omega(\gamma(\alpha_i),\eta(\alpha_i),\W_i)$. We also define $\omega_0$ as the 
constant path at $y_0$. We notice that for all $i\in \lsem 1,n,\rsem$, the loop $
\gamma_{|[\alpha_{i-1},\alpha_i]}*\omega_i*(\eta_{|[\alpha_{i-1},\alpha_i]})^{-1}*\omega_{i-1} $ is homotopically trivial. Therefore, $\eta$ is homotopic to the loop 
$\eta':=\gamma_{|[0,t_0]}*\omega_n*\eta_{|[t_0,1]}$. Consider the lifts $\hat\gamma$
and $\hat\eta'$ of $\gamma$ and $\eta'$ respectively, both starting at $\hat y_0$. They are equal on $[0,t_0]$; since $\gamma$ and $\eta$ both stay in $p(\U)$, the paths $\hat\gamma$ and $\hat\eta'$ stay in the same connected component of $p^{-1}(p(\U))$, hence they both stay in $\U$. We then deduce that $\psi(\eta)=\hat\eta'(1)\in \U$ thus $\psi(\mathcal O)\subset \U$.
Finally, $\psi$ is continuous, then so is $\overline \psi$.

\item Let $\U\subset Y$ open and $K\subset [0,1]$ compact, define 
$\mathcal O := \{ \gamma\in \Omega~|~ \gamma(K)\subset \U \}$. Four cases can occur:
\begin{enumerate}
 \item If $0\in K$ and $y_0 \notin \U $, then $\psi(\mathcal O)= \emptyset$;
 \item If $0\notin K$ or $y_0 \in \U$;
    \begin{enumerate}
     \item If $1 \notin K$, then $\psi(\mathcal O)=\widetilde Y^S$;

     \item If $1\in K\neq [0,1]$ , then $\psi(\mathcal O)= p^{-1}(\U)$;
     \item If $K = [0,1]$ then $\psi(\mathcal O)$ is the connected component of $\hat y_0$ in $p^{-1}(\U)$.
    \end{enumerate}
\end{enumerate}
Case $(1)$ is trivial. In case $(2.a)$, consider $x\in \widetilde Y^S$, $\gamma\in \Omega$ such that $\psi(\gamma)=x$. Take some $t_0\in ]0,1[$ such that $[t_0,1]\cap K = \emptyset$ and consider the path $\eta(t)=y_0$ if $t\leq t_0$ and 
$\eta(t)=\gamma(\frac{t-t_0}{1-t_0})$ if $t\in [t_0,1]$. We indeed have $\psi(\eta)=\psi(\gamma)=x$. A similar argument shows case $(2.b)$. 
Then, in case $(2.c)$, we have $y_0\in \U$ and the connected component $\hat\U$ of $\hat y_0$ in $p^{-1}(\U)$  is path connected. For $x\in \hat\U$, there thus exists some path $\hat\gamma \in \Omega(\hat y_0,x,\hat \U)$ and, for such a path, $\psi(p\circ \gamma)=x$. Furthermore $p\circ \gamma\subset \U$ thus $p\circ \gamma\in \mathcal O$. Therefore, $\hat \U \subset \psi(\mathcal O)$. Certainly, $\psi(\mathcal O)\subset p^{-1}(\U)$ and since $\hat\gamma$ is connected and $\hat\gamma(0)=\hat y_0\in \hat\U$, then $\psi(\mathcal O)\subset \hat \U$.
Finally, $\psi(\mathcal O)=\hat \U$, the connected component of $\hat y_0$ in $p^{-1}(\U)$.

We deduce from the previous analysis that $\psi$ is open, hence $\overline \psi$ is open.
\end{itemize}
Together, the points above show that $\overline \psi$ is an homeomorphism.
\end{proof}

We end this section with two corrolaries of Proposition \ref{prop:spread_fiber}, the proofs of which are left to the reader, and two examples illustrating possible behaviors of the fiber above a branching point. 
\begin{defi} 
 Let $X\xrightarrow{p} Y$ be a  covering branched over $S\subset Y$. 
 Given two paths $\gamma_1,\gamma_2$ with fixed start point $a\in Y\setminus S$ and free end point in an open connected subset $\U\subset Y\setminus S$; $\gamma_1$ and $\gamma_2$ are $p$-homotopic if their lifts, from the same start point, end in the same connected component of $p^{-1}(\U)$.
\end{defi}

  \begin{prop} \label{prop:general_fiber_path}
  Let $X\xrightarrow{p} Y$ be a  covering branched over $S\subset Y$. Let $a\in Y\setminus S$ and let $\U\subset Y$ be open and connected.  Define
$\pi_1(a,p,\U)$ the set of paths with fixed start at $a$ and free end in $\U$ up to $p$-homotopy.

Then, for all $b\in Y$ and all connected neighborhood basis $\boldsymbol \U$ of $b$ we have 
  $$ p^{-1}(b) \simeq \varprojlim_{\U\in \boldsymbol\U} \pi_1(a,p,\U)$$
  where the natural bonding maps $\pi_1(a,p,\U)\rightarrow \pi_1(a,p,\V)$  for $\U\subset \V$.
\end{prop}

\begin{prop} \label{prop:general_fiber_path_universal}
 Let $Y$ be a Hausdorff connected locally path connected topological space and let $S\subset Y$ skeletal such  that $Y\setminus S$ is semi-locally simply connected. Let $a\in Y\setminus S$ and let $\U\subset Y$ be open and connected.   Denote $\widetilde Y^S\xrightarrow{p} Y$ the universal cover of $Y$ possibly branched over~$S$. Define
$\pi_1(a,\U)$ the set of paths with fixed start at $a$ and free end in $\U$ up to homotopy.

Then, for all $b\in Y$ and all connected neighborhood basis $\boldsymbol \U$ of $b$ we have 
  $$ p^{-1}(b) \simeq \varprojlim_{\U\in \boldsymbol\U} \pi_1(a,\U)$$
  where the natural bonding maps $\pi_1(a,\U)\rightarrow \pi_1(a,\V)$  for $\U\subset \V$.
\end{prop}

\begin{example} Consider $Y = \RR^2$ and $S=\{(\frac{1}{n},0)~:~n\in \N^*\cup\{\infty\}\}$. The fiber above  $(0,0)$ in $\widetilde Y^S$ is homeomorphic to $\mathbb Z^\N$ endowed with the weak topology; hence is not locally compact. This answers positively Problem 10.8 of \cite{MR2198590}.
\end{example}
\begin{proof}
 Consider the neighborhood basis given by the discs $D_n$ centered at $(0,0)$ of radius $\frac{1}{n-1/2}$ for $n\in \N^*$. Let $a = (2,0)$ and  $b:=(0,0)$, for $n\in \N^*$, the natural map $\pi_1(a,D_{n+1})\rightarrow\pi_1(a,D_{n})$ has an infinite fiber above every point. The result then follows from Proposition \ref{prop:general_fiber_path_universal}.
\end{proof}

\begin{example} Consider $Y = \RR^2$ and $S=\{(\frac{1}{n},0)~:~n\in \N^*\cup\{\infty\}\}$ and denote by $\Gamma$ the absolute Galois group $\Gamma(\widetilde Y^S/Y)$. 
Let $\Gamma_2:=\langle \Gamma[\gamma_n]^2\Gamma~:~ n\in \N \rangle$ where $\gamma_n$ is a simple loop around $(1/n,0)$. Consider $Y_2$, the completion of the spread $\Gamma_2\backslash\ord(\widetilde Y^S) \rightarrow Y$.
The fiber above $(0,0)$ in  $Y_2$ is  a Cantor set. This answers positively Problem 10.7 of \cite{MR2198590}.
\end{example}
\begin{proof}
 Apply Proposition \ref{prop:general_fiber_path} with the same neighborhood basis are the previous example. Then use the fact that the sets $\pi_1(a,p,D_n)$ have cardinal $2^n$ and the bonding maps are all or order 2.
\end{proof}

\subsection{Galoisian branched covering}

The usual notion of Galoisian covering can be extended in a natural way to branched covering via the group of automorphisms of the branched covering. Although, as we shall see, the completion of a Galoisian covering need not be Galoisian. For instance the universal  covering possibly branched above some locus is not always Galoisian, however we provide a topological criteria akin to semi-local simple connectedness which ensure it is. It allows to state a Galoisian correspondance for branched covering.

\begin{defi} A branched covering $X\xrightarrow{p} Y$ is Galoisian if the group $\Gamma(X/Y)$ of automorphism of $p$ acts transitively on the fibers of $p$. It is called quasi-Galoisian if $p$ is the completion of a Galoisian covering.

\end{defi}

\begin{rem} If a branched covering $X\xrightarrow{p} Y$ is Galoisian then $p$ induces an homeomorphism $\Gamma(X/Y)\backslash X \rightarrow Y$. 
\end{rem}

\begin{lem} Let $X\xrightarrow{p} Y$ be a branched covering, the group of automorphisms of $p$ is exactly the group of automorphisms of the induced branched covering $p^{-1}(\U) \rightarrow \U$ for any $\U$ portly in $Y$:
$$\Gamma(X/Y) = \Gamma(p^{-1}(\U)/\U)$$
    
\end{lem}
\begin{proof}
 An automorphism $\phi$ of $X \rightarrow Y$ lifts the identity thus  for any $\U$ subset  of $Y$, the automorphism $\phi$ preserves $p^{-1}(\U)$. In particular, for any $\U$ portly subset of $Y$, such a $\phi$ induces an automorphism of $p^{-1}(\U) \xrightarrow{~~~} \U$. Furthermore, with $\U$ portly in $Y$, by functoriality of the spread completion and since the completion of $p^{-1}(\U)\rightarrow Y$ is $X\rightarrow Y$, an automorphism of $p^{-1}(\U) \xrightarrow{~~~} \U$ extends uniquely to a automorphism of $X \xrightarrow{~~~} Y$.   Therefore, $\Gamma(X/Y) = \Gamma(p^{-1}(\U)/\U))$.
\end{proof}
\begin{cor}Let $X\xrightarrow{p} Y$ be a branched covering, then
 $\Gamma(X/Y) =  \Gamma(\ord(X)/\ord(Y))$.
\end{cor}

The universal covering branched over some locus need not be Galoisian. Indeed, one can consider $Y=\RR^2$ and $S:=\{(1/n,0):n\in \N^*\}\cup \{(0,0)\}$. On the one hand, the automorphisms group of $\widetilde Y^S$ is countable (it is the fundamental group of $Y\setminus S$ thus a free group generated by countably infinitely many generators). On the other hand, there are uncountably infinitely many homotopy classes in of path from $(2,0)$ to $(0,0)$ and none of these classes are equivalent in the sense of Proposition \ref{prop:ramified_covering_path};
hence the fibre above $(0,0)$ is uncountable. The action of $\Gamma(X/Y)$ is then not transitive on the fiber above $0$ and $\widetilde Y^S \rightarrow Y$ is not Galoisian.

\begin{defi} Let $X$ be a Hausdorff, connected, locally path connected topological space.  We say that an open subset $\U$ is semi-locally simply connected at $x \in \overline \U$ if there exists a connected open neighborhood $\V$ of $x$ such that every loop $\gamma$ of $\Omega(x,\V\cap \U)$ is almost trivial. 

$\U$ is then semi-locally simply connected in $X$ if $\U$ is dense and semi-locally simply connected at every point of $X$
\end{defi}
\begin{rem}
 The definition of semi-locally simply connected is coherent witht the usual definition of semi-locally simply connected. Indeed a topological space $X$ is semi-locally simply connected if and only if it is semi-locally simply connected in  itself.
\end{rem}

\begin{rem} \label{rem:simple_cases} Using the same notation as in the definition above, there are generic situations in which dense open subsets are semi-locally simply connected at every point of the whole space.
\begin{itemize}
 \item If $X$ is a manifold and $X\setminus \U$ is a codimension at least 2 submanifold, more generally if $X$ is a simplicial complex and $X\setminus \U$ is a union of facets (such that $\U$ is still dense and locally connected);
 \item  or if $X$ admits a basis of neighborhoods such that for each $\V$ of said basis, $\V\cap \U$ is semi-locally simply connected $\V$;
\end{itemize}
then $\U$ is locally simply connected at every point of $X$. 
\end{rem}

While considering the universal covering $\widetilde Y^S$ of some $Y$ possibly branched over some subset $S$, there is an ambiguity: the branching locus of $\widetilde Y^S$ might be smaller than $S$. One can consider the example where $Y=\mathbb S^2$ the 2-dimensionnal sphere and $S$ is a singleton, since the complement of a point is simply connected, the universal covering of $\mathbb S^2$ possibly branched at some $y\in Y$ is itself and the branching locus is empty.

\begin{lem}\label{lem:dense_orbits}
 Let $X\xrightarrow{p} Y$ be a quasi-galoisian branched covering. Let $a\in Y$ and $\hat a \in p^{-1}(a)$, the orbit  $\Gamma(X/Y)\hat a$  is dense in $p^{-1}(a)$.
\end{lem}
\begin{proof} We use the same notation as in Proposition \ref{prop:spread_fiber}. 
  Since the branched covering $p$ is quasi-Galoisian, for every $\U\in \boldsymbol\U$, $\Gamma(X/Y)$ acts transitively on $X_\U$. Furthermore, $\Gamma(X/Y)$ acts by homeomorphism on $X$, hence preserves connected components, its actions thus commutes with the bonding maps of the projective system $(X_\U)_{\U\in \boldsymbol{\U}}$. Finally, the action of $\Gamma(X/Y)$ on each $X_\U$ lifts to an action on $\varprojlim_\U X_\U$ whose orbits are dense.
\end{proof}

\begin{prop}\label{cor:Galoisian_fiber}
  Let $X\xrightarrow{p} Y$ be a quasi-Galoisian branched covering possibly branched over $S\subset Y$.
  
  For all $b\in Y$ and all connected neighborhood basis $(\U_i)_{i\in I}$ of $b$ we have 
  $$ p^{-1}(b) \simeq \varprojlim_{i\in I} \Gamma(X/Y)/\Gamma_{\hat\U_i} $$
  where $\hat\U_i$ is the connected component of $\hat b$ in $p^{-1}(\U_i)$ for $i\in I$.
  
\end{prop}

\begin{proof}
  Since $X\xrightarrow{p} Y$ is quasi-Galoisian, $\Gamma(X/Y)$ acts transitively on the connected components of $p^{-1}(\U)$ for every open $\U\subset Y$. The group $\Gamma_\V$ is by definition the stabilizer of $\V$ for $\V$ a connected of $p^{-1}(\U)$  for some open $\U\subset Y$. The discrete spaces $\Gamma(X/Y)/\Gamma_{\hat\U}$ and $X_\U$ are thus homeomorphic (using the same notations as Proposition \ref{prop:spread_fiber}). The result then follows from Proposition \ref{prop:spread_fiber}.

  \end{proof}

\begin{lem} \label{lem:alternative_countable}
 Let $X\xrightarrow{p} Y$ be a quasi-Galoisian branched covering possibly branched over $S\subset Y$ and let $b\in Y$. 
 
 Assume that $\Gamma(X/Y)$ is countable, then the following are equivalent:
 \begin{enumerate}[(i)]
  \item $p^{-1}(b)$ is discrete;
  \item $p^{-1}(b)$ is countable;
  \item the action of $\Gamma(X/Y)$ on $p^{-1}(b)$ is transitive.
 \end{enumerate}

\end{lem}
\begin{proof}
 Consider a neighborhood basis $(\U_n)_{n\in \N}$ of $b$ in $Y$ indexed over $\N$ such that $\U_n\subset \U_m$ whenever $n\geq m$. Denote by $X_n$ the discrete space of connected components of $p^{-1}(\U_n)$ and $(p_{n,m})_{n\geq m}: X_n \rightarrow X_m$ the bonding maps which associated a connected component $\V$ of $p^{-1}(\U_n)$
 its connected component in $p^{-1}(\U_m)$.
 For $n\in\N$, define $N_n := \#p_{n+1,n}^{-1}(x)$ for some $x\in X_n$. Since the action of $\Gamma(X/Y)$ is transitively on each $X_n$ and commutes with the bonding maps, $N_n$ does not depends on the choice of $x$.
 \begin{itemize}
  \item  If $N_n>1$ for infinitely many $n\in \N$, then the projective limit $\varprojlim_n X_n$ is uncountable, hence the action of $\Gamma(X/Y)$ is not transitive. Furthermore, a basis of the topology is given by the inverse image of   subsets of the $X_n$ by the natural projection $p_n : \varprojlim_n X_n \rightarrow X_n$. None of these preimage is reduced to a point, hence the topology of $\varprojlim_n X_n$ is not discrete.
  \item  If $N_n> 1$ only for finitely many $n\in \N$, then the projective limit  $\varprojlim_n X_n$ is  stationnary. Then $\varprojlim_n X_n \simeq X_m$ for some $m\in \N$ big enough. In particular $\varprojlim_n X_n$ is discrete and the action of $\Gamma(X/Y)$ on it is transitive, since $\Gamma(X/Y)$ is countable then so is $\varprojlim_n X_n$. 
 \end{itemize}

\end{proof}

\begin{prop} \label{prop:Galoisian_semilocal}
Let $Y$ be a connected, locally path connected, second countable, Hausdorff topological space and let $S\subset Y$ skeletal such that $Y\setminus S$ is semi-locally simply connected.
Let $\widetilde Y^S\xrightarrow{p} Y$ be the universal covering of $Y$ possibly branched  above $S$. The following are equivalent:
\begin{enumerate}[(i)]
\item $\widetilde Y^S\xrightarrow{p} Y$  is Galoisian;
 \item the branched covering $\widetilde Y^S\xrightarrow{p} Y$ has discrete fibers;
  \item $Y\setminus S$ is semi-locally simply connected in $Y$;
\end{enumerate}
\end{prop}
\begin{proof}\

  \begin{itemize}

   \item $(ii)\Rightarrow (i)$. The fibers of $p$ are discrete and by Lemma \ref{lem:dense_orbits} the orbits of $\Gamma(\widetilde Y^S/Y)$ in the fibers of $p$ are dense. Therefore, $\Gamma(\widetilde Y^S/Y)$ acts transitively on the fibers.
   \item $(i)\Rightarrow (ii)$. Since $Y$ is second countable so is $Y\setminus S$; then  $Y\setminus S$ is second countable and semi-locally simply connected hence its fundamental group is countable and so is $\Gamma(\widetilde Y^S/Y)$. The hypotheses of Lemma \ref{lem:alternative_countable} are satisfied. By hypothesis, the action of $\Gamma(\widetilde Y^S/Y)$  on the fibers of $p$ is transitive, then the fibers of $p$ are discrete.
   \item $(iii)\Rightarrow (ii)$. Assume that $\ord(Y)$ is semi-locally simply connected at every point of $Y$ and consider some $y\in Y$. If $y\in \ord(Y)$, there is nothing to prove; henceforth we assume $y\notin \ord(Y)$ and consider some connected open neighborhood $\U$ of $y$ such that the constant loop at $y$ is in the adherence of every homotopy classes in $\Omega(y,\ord(Y))$ of loops in $\Omega(y,\ord(\U))$. We choose a connected component $\V$ of $p^{-1}(\U)$ and consider $x_1,x_2 \in \V$. Consider a path $\hat\gamma:[0,1]\rightarrow X$ from $x_1$ to $x_2$ such that $\hat\gamma(]0,1[)\subset \ord(X)\cap \V$. Such a path exists since $\ord(X)$ is open dense and locally connected in $X$ hence $\ord(X)\cap \V$ is open dense in $\V$ and path connected. By hypothesis, for every open subset $\W$ containing $y$ there exists Consider some open neighborhood $\W$ of $y$, by hypothesis we can choose some $\gamma_1\in \Omega(y,\W)$ homotopic to $\gamma_0=p\circ \hat\gamma$ in $\Omega(y,\ord(Y)$.  Let $H:[0,1]\times[0,1]\rightarrow Y$ be an homotopy from $\gamma_0$ to $\gamma_1$ in $\Omega(y,\ord(Y))$ the ordinary part of $H$ defined on 
   $]0,1[\times [0,1]\rightarrow Y$ lifts to $\ord(X)$ and, by Lemma \ref{lem:complete_path_carac}, this lift extends continuously to a lift $\widetilde H:[0,1]\times [0,1]\rightarrow X$ of $H$. Since the fibers of $p$ are totally disconnected and $\forall s\in [0,1], H(0,s)=H(1,s)=y$, then $\widetilde H(0,\cdot)$ and $\widetilde H(1,\cdot)$ are constant. In particular, $x_1$ and $x_2$ are in the same connected component in $p^{-1}(\W)$. Since the connected components of $p^{-1}(\W)$ generate the topology of $X$ as $\W$ go through the open of $Y$, and since $X$ is Hausdorff, then $x_1=x_2$.
   Finally, $\V\cap p^{-1}(y)$  is a singleton.

 \item $(ii)\Rightarrow (iii)$. Assume the fibers of $p$ are discrete and consider some neighborhood $\V$ of $x\in X$ such that $\V\cap p^{-1}(p(x))=\{x\}$. Let $\U$ be some connected neighborhood of $p(x)$ such that the connected component $\W$ of $x$ in $p^{-1}(\U)$ is a subset of $\V$. 
 By Lemma \ref{lem:spread_connected_image}, we have  $p(\W)=\U$; consider some path $\gamma \in \Omega(y,\U)$ and some lift $\hat\gamma : [0,1]\rightarrow X$. Since $\gamma(0)=\gamma(1)=y$ and since $\V \rightarrow \U$ is a complete spread, both $\hat\gamma(0)$ and $\hat \gamma(1)$ are in $p^{-1}(y)\cap \U$ and are thus equal to $x$ ie $\hat\gamma \in \Omega(x,\V)$. 
 Let $\W$ be some neighborhood of $y$, choose some $t_0,t_1 \in ]0,1[$ such that $\hat\gamma(]0,t_0])\subset p^{-1}(\W)$ and $\hat\gamma([t_1,1[)\subset p^{-1}(\W)$. The connected component $\hat\W$ of $x$ in $p^{-1}(\W)$ contains both pieces of $\hat\gamma$ considered and $\hat\W\cap \ord(X)$ is path connected; there thus exists a path $\omega:\hat\gamma(t_1)\rightsquigarrow \gamma(t_0)$ in $\hat\W$. Since $\ord(X)$ is the universal covering of $\ord(Y)$, it is simply connected and the loop $\gamma_{|[t_0,t_1]}\circ \omega$ is trivial. Therefore, $\hat\gamma$ is homotopic to $\hat\gamma_{[t_1,1]}\circ \omega \circ \hat\gamma_{[,t_0]}$.  Hence, $\hat\gamma$ is homotopic in $\Omega(x,\ord(X))$ to a loop in $\Omega(x,\ord(\hat\W)$. Finally, $\gamma$ is homotopic in $\Omega(y,\ord(Y))$ to a loop in $\Omega(t,\ord(\W))$. The open subset $\W$ is arbitrary, then $\ord(Y)$ is semi-locally simply connected at $y$.
 
  \end{itemize}

\end{proof}

The following Theorems are then a direct consequence of Propositions \ref{prop:Galoisian_semilocal} and \ref{prop:universality}.
\begin{thm}[Galois correspondance, general case]
 Let $Y$ be a connected, locally path connected Hausdorff topological space and let $S\subset Y$ skeletal.
 
 Every covering $X$ of $Y$ possibly branched over $S$ is isomorphic to the completion of the quotient of $\ord(\widetilde{Y}^S)$ by a subgroup of $\Gamma(\widetilde{Y}^S/Y)$. 
 
 Finally, this subgroup is normal if and only if $X$ is quasi-Galoisian.
\end{thm}
\begin{thm}[Galois correspondance, good case]
 Let $Y$ be a connected, locally path connected Hausdorff topological space and let $S\subset Y$ skeletal.
 Assume $Y\setminus S$ is semi-locally simply connected in $Y$. 
 
 Then the universal covering $\widetilde Y^S$ possibly branched over $S$ is Galoisian. 

 Furthermore, every covering $X$ of $Y$ possibly branched over $S$ is isomorphic to the quotient of $\widetilde{Y}^S$ by a subgroup of $\Gamma(\widetilde{Y}^S/Y)$. 
 Finally, this subgroup is normal if and only if $X$ is Galoisian.
\end{thm}

\begin{example}
We construct a finite Galoisian covering of the 3-sphere branched over a wild Cantor.

  Let $Y := \mathbb S^3 $ be the 3-sphere and $S$ be Antoine's Necklace \cite{general_antoine} ie a wild Cantor set. One easily checks that $Y\setminus S$ is not semi-locally simply connected in $Y$ hence  $\widetilde{Y}^S$ is not Galoisian. 
  From \cite{general_antoine}, there exists a surjective representation $\rho : \Gamma(\widetilde{Y}^S/Y) \rightarrow \mathfrak{A}_6$, its kernel $K$ corresponds to a finite Galoisian branched covering $X\xrightarrow{p} Y$ of Galois group $\mathfrak{A}_6$.
  
  Consider any neighborhood $\U$ of some $s\in S$ and $\hat\U$ a connected component of $p^{-1}(\U)$ in $X$. The restriction $\hat\U\xrightarrow{p}\U$ is a Galoisian branched covering whose Galois group is the set-wise stabilizer of $\hat\U$. An inspection of Blankinship construction shows this stabilizer is the whole group and then $\Gamma(\hat\U/\U)=\Gamma(X/Y)\simeq \mathfrak{A}_6$. As a consequence, $X\xrightarrow{p}Y$ is branched over $S$. 
  
\end{example}
The previous example is related to an open question of Montesinos on the existence of a covering of the sphere whose branched locus is a wild cantor set such that the total space  is a manifold, the conjecture state such a covering does not exists. We do not prove that the branched covering constructed is indeed a manifold, it is however a good candidate if the conjecture is false.

%
%
%
%
%
%
%
%
%
%
%

\section{Singular (G,X)-manifolds}
We give ourselves some analytical structure $(G,X)$, ie a connected, locally connected, Hausdorff topological space $X$ together with a faithful action of a group $G$ by homeomorphism such that for all $\phi_1,\phi_2\in G$ and all non trivial open subset $\U\subset X$, if the action of $\phi_1$ and $\phi_1$ agree on $\U$ they agree on $X$. A $(G,X)$-manifold is a second countable Hausdorff topological space with an atlas whose of model space $X$ and whose change charts are restrictions of the action of some element of $G$. 

The simplest example of singularities occuring in the litterature of $(G,X)$-manifolds are conical singularities in homogeneous riemannian manifolds. Such manifolds with conical singularities have an underlying metric space structure that gives a natural notion of isomorphism class. In a way, the author feels that such an underlying structure distracts from the fact that the notion of isomorphism in singular $(G,X)$-manifold does not depend is actually more general. The language we present in this section ease the analysis of  singularities in manifolds in particular those which do not admit a natural metric structure such as Lorentzian or projective $(G,X)$-manifolds. As illustration of the richness of non metric singularities and their usefulness we wish to refer for instance to \cite{Particules_1} for a fundamental study of "natural" 2+1 Lorentzian singularities or \cite{MR3190306,MR3814342} for their role in geometric transitions. 

We begin by defining a notion of singular $(G,X)$-manifolds, which encompass most singularities of the litterature, and we show some elementary properties ensuring the theory of such manifold works well. The we generalize the developpement Theorem of $(G,X)$-manifolds to singular $(G,X)$-manifolds.

\subsection{Preliminaries on singular $(G,X)$-manifolds}
A $(G,X)$-preatlas of $N$ is a $(G,X)$-atlas on an opens subset $M\subset N$.
\begin{defi} A subset $\U$ of a topological space $M$ is  portly if it is open, dense and locally connected in $M$.
 
\end{defi}

\begin{defi} An almost everywhere $(G,X)$-atlas on a topological space $M$ is a $(G,X)$-atlas defined on some portly subset of $M$. 
\end{defi}

    \begin{lem} \label{lem:GX_atlas_union}
        Let $N$ be a locally connected topological space.
        Let $(\B_k)_{k\in K}$ a family of a.e.$(G,X)$-atlases on $N$. Assume for  all $(k,k')\in K^2$, there exists an a.e $(G,X)$-atlas $\mathcal A$ such that 
        \begin{itemize}
         \item the support of $\A$ is included in both supports of $\B_{k}$ and $\B_{k'}$;
         \item $\A$ is thinner than both  restrictions of $\B_{k}$ and $\B_{k'}$.
        \end{itemize}

            Then the union $\bigcup_{k\in K} \B_k$ is an a.e.$(G,X)$-atlas.
    \end{lem}
					\begin{proof}
					Let $k,k'\in K$ and write $\B_k=(\U_i,\V_i,\phi_i)_{i\in I(k)}$ and
					$\B_{k'}=(\U_i,\V_i,\phi_i)_{i\in I(k')}$.
					Let $i\in I(k), j\in I(k')$ such that $\U_i\cap \U_j \neq \emptyset$ and let $\widehat\W \subset\U_i\cap\U_j$ be an open connected non empty subset.
					Let $\A$ be an a.e.$(G,X)$-atlas defined on some portly subset $M\subset N$ which is thinner that both $\B_{k|M}$ and $\B_{k'|M}$.
					
					\begin{itemize}
					 \item Consider $p\in \widehat\W\cap M$ and a chart $(\U,\V,\phi)$ of $\A$ around $p$
					and let $\W \subset \widehat \W\cap \U$ be an open connected neighborhood of $p$.
					there exists $g,h\in G$ such that
					$$\phi_{|\W} =g\circ \phi_{i|\W} \quad \quad  \phi_{|\W} =h\circ \phi_{j|\W},$$
					hence
					$$\phi_{j|\W}=(h^{-1} g)\circ \phi_{i|\W}.$$
					 Define $g_{p,\U,\W}:=h^{-1}g$ so that $\phi_{j|\W}= g_{p,\U,\W}\circ \phi_{i|\W}$.
				
					\item  We show that $g_{p,\U,\W}$ depends neither on $\U$ nor on $\W$ nor on $p$.
					
					Fixing $p\in \widehat \W\cap M$, consider
					 $(\U,\V,\phi)$ and $(\U',\W',\phi')$ two charts of $\A$ on a neighborhood of $p$ 
					as well as $\W$ and $\W'$ two open connected neighborhoods of  $p$ in $\widehat \W \cap \U$ and $\widehat \W \cap \U'$ respectively.
					The restrictions of $g_{p,\U,\W}$ and $g_{p,\U',\W'}$ to $\phi_i(\W\cap \W')\subset X$ thus $g_{p,\U,\W}=g_{p,\U',\W'}$.
					and $g_p:=g_{p,\U,\W}$ only depends on $p$. 
					Furthermore, if $q\in \W$, we prove the same way that
					$g_{q,\U,\W}=g_{p,\U,\W}$.
					Finally, the map $\W\cap M\rightarrow G, p\mapsto g_p$ is locally constant.
					Since $M$ is portly and since $\widehat \W$ is connected, the intersection $\widehat \W \cap M$ is connected and $p\mapsto g_p$ is constant on $\widehat \W\cap M$.
					
					\item We proved there exists $g\in G$  such that
					$$\forall x\in M\cap \widehat \W, \quad  \phi_j(x)=g\circ \phi_i(x)$$
					Since $\phi_i,\phi_j$ and $g$ are continuous, $X$ is Hausdorff and $M$ is dense, 
					the intersection $\widehat \W\cap M$ is then dense in 
					$\widehat \W$ and
					$$\phi_{j|\widehat\W}=g\circ \phi_{i|\widehat\W}.$$
					\end{itemize}
					Finally, $\B_k\cup \B_{k'}$ is an a.e. $(G,X)$-atlas, moreover  $k,k'\in K$ are arbitrary thus $\bigcup_{k\in K} \B_k$ is an a.e. $(G,X)$-atlas.
					
					\end{proof}

					\begin{cor} \label{cor:maximal_atlas}
					Let $M$ be a locally connected topological space. 
					Consider the set $\mathcal E$ of couples $(\U,\mathcal A)$ where $\mathcal A$ is an a.e $(G,X)$-structure defined on $\U\subset M$. 
					The set $\mathcal E$ is ordered in the following way $(\U,\mathcal A) \leq (\U',\mathcal A')$ if $\U\subset \U'$ and $\mathcal A'_{|\U}$ is thinner that $\mathcal \A$. 
					
					Then, for every $(\U,\mathcal A) \in \mathcal E$ the set of elements of $\mathcal E$ greater than $(\U,\mathcal A)$ has a maximum.
					\end{cor}
                    \begin{proof}
                      By Lemma \ref{lem:GX_atlas_union}, it suffices to take the union of all the a.e. $(G,X)$-atlases greater than the given $(\U,\mathcal A)$.
                    \end{proof}					
                    The above Corollary justifies the importance of portly subsets and the choice of terminology "almost everywhere". The $(G,X)$-structure on a "negligible" set is can be reconstructed from its complement.
                    \begin{defi} An a.e. $(G,X)$-structure on a topological space $M$ is  a.e. $(G,X)$-atlas on $M$.                     
                    \end{defi}
                    
					\begin{defi} A singular $(G,X)$-manifold is a locally connected $T_2$ second countable topological space $M$ endowed with an a.e. $(G,X)$-structure maximal in the sense of Corollary \ref{cor:maximal_atlas}.
					\end{defi}

					\begin{defi}[a.e.$(G,X)$-morphism] Let
					$M$ and $N$ be two singular $(G,X)$-manifolds.

					An a.e.morphism $\phi:M\rightarrow N$ is a continuous map such that there exists portly subsets $\U\subset \reg(M)$ and $\V\subset \reg(N)$ such that
					$$\phi_{|\U}^{|\V}:(\U,\A_\U)\rightarrow (\V,\B_\V)$$ is a $(G,X)$-morphism.
					\end{defi}

					\begin{rem} Composition of a.e. $(G,X)$-morphisms is an a.e. $(G,X)$-morphisms. The category $\CatSingGX$ of singular $(G,X)$-manifold which morphisms are a.e $(G,X)$-morphisms is then well defined. 
					An isomorphism of singular $(G,X)$-manifolds is then a homeomorphic a.e. $(G,X)$-morphism.
					\end{rem}

                    \begin{rem} A.e. $(G,X)$-morphisms are analytical: if two are equal on an open set, they are equal everywhere.
                     
                    \end{rem}

					\begin{rem}\label{prop:GX_sing_morph_local}
					Let $M\rightarrow N$ a continuous map between singular $(G,X)$-manifolds.
					The following propertie are equivalent.
					\begin{enumerate}[(i)]
					 \item $f$ is an a.e. $(G,X)$-morphism;
					 \item $f$ is locally an a.e. $(G,X)$-morphism ;
					 \item $f_{|\U}$ is an a.e. $(G,X)$-morphisms for some portly set $\U\subset M$.
					\end{enumerate}
					\end{rem}

                \begin{lem} 
                   Let $M$ and $N$ be singular $(G,X)$-manifolds and let $f:M\rightarrow N$ an a.e $(G,X)$-morphism. 
                   The set $\mathcal E$ of open subset $\U$ such that $\U\subset \reg(M)$ and $f$ induces a $(G,X)$-morphism $\U\rightarrow \reg(N)$.   
                   
                   Ordered by the inclusion, $\mathcal E$ has a maximum.
                    
                \end{lem}
                \begin{proof}
                  Let $\U_0:=\bigcup_{\U\in \mathcal E}\U$, the map $f_{|\U_0}^{|\reg(N)}$ is locally a $(G,X)$-morphism hence a $(G,X)$-morphism. Finally, $\U_0\in \mathcal E$ and is certainly the maximum.
                \end{proof}

                \begin{defi} Let $M$ and $N$ be singular $(G,X)$-manifolds and let $f:M\rightarrow N$ an a.e $(G,X)$-morphism.
                The maximal open subset of $M$ on which $f$ induces a $(G,X)$-morphism is the regular locus of $f$ which we denote by $\reg(f)$.
                \end{defi}

                    \begin{prop}
                     Let $M$ be a singular $(G,X)$-manifold and let $f:M\rightarrow X$ an a.e $(G,X)$-morphism. Then, 
                     $\reg(f)=\reg(M)$
                    \end{prop}
                    \begin{proof}
                     Let $x\in \reg(M)$ and let $(\U,\V,\varphi)$ be a connected chart around $x$. The map $f\circ\varphi^{-1}$ induces a $(G,X)$-morphisms from $\V\cap \varphi(\reg(f))$ to $X$. Since $\reg(f)$ is locally connected in $M$ and $\U$ is connected, $\V\cap \varphi(\reg(f))$ is also connected and there exists a unique $g\in G$ such that $$\forall y\in \V\cap \varphi(\reg(f)),\quad  f\circ \varphi^{-1}(y)=g(y).$$
                     Since $\varphi(\Reg(f))$ is dense in $\V$ and $f\circ \varphi^{-1}$ is continuous, we have
                     $$\forall y\in \V,\quad  f\circ \varphi^{-1}(y)=g(y).$$
                     Finally, $f$ is a regular on $\U$. The result follows.
                    \end{proof}

                \begin{prop}{}\label{prop:GX_sing_homeo_reg}
						 Let $M$ and $N$ be singular $(G,X)$-manifolds and let $f:M\rightarrow N$ an a.e $(G,X)$-morphism. Then,					
					$$\reg(f)=\reg(M)\cap \{x\in M  ~|~ f \text{ is a local homeomorphism around }x\}$$
					\end{prop}
					\begin{proof}
					By definition, $\reg(f)\subset \reg(M)$.
					Consider $x\in \reg(f)$  since $f$ is a $(G,X)$-morphism on a neighborhood of $x$, in particular it is a local homeomorphism. 
					Hence, $\reg(f)\subset\reg(M)\cap \{x\in M  ~|~ f \text{ is a local homeomorphism around }x\}$.
					Consider $x\in \reg(M)$ around which $f$ is a local homeomorphism and consider an open neighborhood $M_0\subset \reg(M)$ and an open neighborhood $N_0\subset \reg(N)$ such that $g:=f_{|M_0}^{|N_0}$ is an homeomorphism. Without loss of generality, we can assume $M_0\subset \reg(M)$.
					On the one hand, $N_0\cap \reg(N)$ comes with a  $(G,X)$-atlas induced by the $(G,X)$-structure on $\reg(N)$. On the other hand, another $(G,X)$-atlas is given by the pushforward by $g$ of the $(G,X)$-atlas of $M_0$.  
					Since $g$ is a $(G,X)$-morphism from a portly subset of $M_0$ to a portly subset of $N_0$ for both structures, they are then compatible. Furthermore, $g$ is $(G,X)$-morphism for the pushforward $(G,X)$-structure on $N_0$. Hence, the $(G,X)$-atlas of $\reg(N)$ extends to $\reg(N)\cup N_0$ and $g$ is a $(G,X)$-morphism. By maximality, $\reg(N)$ then contains $N_0$ and $f$ induces a $(G,X)$-morphism $M_0\rightarrow N_0$. Then, $M_0\subset \reg(f)$.
					\end{proof}

					\begin{cor}
					   Both regular and singular locii are preserved by the a.e. $(G,X)$-morphisms which are local homeomorphisms.
					\end{cor}
					
					\begin{rem} As a consequence, 
					 an isomorphism $M\xrightarrow{f} N$ of singular $(G,X)$-manifolds is an isomorphism of almost everywhere $(G,X)$-atlas in the sense that given the maximal almost everywhere $(G,X)$-atlases $\mathcal A$ and $\mathcal B$ of $M$ and $N$ respectively, the pullback $f^*\mathcal B$ and $\mathcal A$ are equal. It is then an isomorphism in every natural ways relative to singular $(G,X)$-manifolds.

                    \end{rem}
                    \begin{prop} Let $M$ be a singular $(G,X)$-manifold and let $N\xrightarrow{\pi} M$ be a branched covering.
                    
                    $N$ has a unique a.e. $(G,X)$-structure for which $\pi$ is an a.e. $(G,X)$-morphism.
                    and $\reg(\pi)=\reg(N)\cap \ord(\pi)$.
                  
                     Furthermore, the Galois group $\Gamma(\widetilde M^S/M)$ acts by a.e. $(G,X)$-morphisms.
                    \end{prop}
                    \begin{proof}
                    Let $\B$ be the a.e. $(G,X)$-atlas on $M$.
                     $\pi$ induces a covering $\hat\U:=\pi^{-1}(\ord(M)\cap \reg(M))\xrightarrow{\pi} \ord(M)\cap \reg(M)=:\U$ thus the pull back of $\B$ by $\pi_{|\hat\U}$ makes $\pi_{|\hat\U}^{|\U}$ into $(G,X)$-morphism. Since $\ord(M)$ and $\reg(m)$ are portly then $\U$ is portly in $M$ thus in $\ord(M)$, therefore $\hat \U$ is portly in $\ord(N)$ and portly in $N$. The $(G,X)$-structure on $\hat U$ thus induces a unique a.e. $(G,X)$-structure on $N$ and $\pi$ is an a.e. $(G,X)$-morphism with respect to this $(G,X)$-structure.               
                     The statement on $\reg(\pi)$ is then a direct consequence of \ref{prop:GX_sing_homeo_reg}.
                     
                     Consider $\A$ such an a.e. $(G,X)$-structure as well as $\hat \U$ and $\U$ as above. Since $\pi_{|\hat\U}^{|\U}$ is an a.e. $(G,X)$-morphism and a local homeomorphism the restriction of $\A$ to $\hat\U$ is the compatible thus equal to the pull-back of $\B$ by $\pi_{|\hat\U}^{|\U}$. Finally, $\hat\U$ is portly in $N$ thus $\A$ is the a.e. $(G,X)$-structure constructed above.
                     
                     The Galois group $\Gamma(N/M)$ acts by homeomorphisms which restriction to  $\reg(N)\cap \ord(\pi)$  are $(G,X)$-morphisms. Since  $\reg(N)\cap \ord(\pi)$  is portly, then $\Gamma(N/M)$ acts by a.e. $(G,X)$-morphisms.
                     
                    \end{proof}

					\begin{cor} Let $M$ be a singular $(G,X)$-manifold and let $S$ be a skeletal subset of $M$ containing $\sing(M)$. Then $\widetilde M^S$ admits a unique a.e. $(G,X)$-structure such that the natural projection $\widetilde M^S\rightarrow M$ is an a.e. $(G,X)$-morphism.
					
					Furthermore, the Galois group $\Gamma(\widetilde M^S/M)$ acts by a.e. $(G,X)$-morphisms. 
					 
					\end{cor}

\subsection{Developping map}

For brievety sake, we will denote by $\widetilde M$ the universal covering of $M$ possibly branched over $\sing(M)$ which can also be denoted by $\widetilde M^{\sing(M)}$ in the notation of the previous section.

In the whole section $(X_\alpha)_{\alpha \in A}$ is a familly of singular $(G,X)$ model spaces ie connected singular $(G,X)$-manifolds.

%

\begin{defi}[developping map]
  Let $N\xrightarrow{p} M$ be a branched covering of singular $(G,X)$-manifold.
  A developping map $\D$ of $N\rightarrow M$ is any a.e. $(G,X)$-morphism $N \rightarrow X$. 
  
  For simplicity sake, a developping map of $\widetilde M\rightarrow M$ is simply called a developping map of $M$.
  \end{defi}

The proofs of the following two Lemmas are identical to the usual proof in the regular context and are thus skipped.  
\begin{lem} \label{lem:holonomy} Let $N\xrightarrow{p} M$ be a quasi-Galoisian branched covering of singular $(G,X)$-manifold. If $\D$  is a developping map of $N\xrightarrow{p} M$. 
Then there exists a unique morphism $\rho:\Gamma(N/M)\rightarrow G$ such that $\D$ is $\rho$-equivariant.
\end{lem}
\begin{defi} 
 Let $N\xrightarrow{p} M$ be a quasi-Galoisian branched covering of singular $(G,X)$-manifold admitting a developping map $\D$. The holonomy of $N\rightarrow M$ (associated to $\D$) is the morphism $\rho$ of Lemma \ref{lem:holonomy}.
\end{defi}
 \begin{rem}
    A developping map $\D$ of a quasi-Galoisian covering $M\rightarrow N$ sends points to fix points of their stabilizer : $\D(x) \in \fix(\rho(\Gamma_x))$.
  \end{rem}
  
\begin{rem} 
 Let $M$ be a singular $(G,X)$-manifold. The holonomy of $\widetilde M \rightarrow M$ is indeed the holonomy of $\reg(M)$ in the usual sense.
\end{rem}
\begin{lem}
 Let $N\xrightarrow{p} M$ be a quasi-Galoisian branched covering of singular $(G,X)$-manifold. If it admits a couple $(\D,\rho)$ of developping map and holonomy, then all other such couples are obtained via conjugation ie:
 if $(\D',\rho')$ is another couple developping map and holonomy, then there exists a unique $g\in G$ such that $(\D',\rho)= (g\circ\D,\rho^g)$ with $\rho^g : \gamma \mapsto g\rho(\gamma)g^{-1}$.
\end{lem}

\begin{lem}\label{lem:descend_developing_map} Let $M\rightarrow N$ be a quasi-Galoisian branched covering of singular $(G,X)$-manifolds admitting a developping map. If the holonomy of $M\rightarrow N$ is trivial, then the developping map of $M\rightarrow N$ induces an a.e. $(G,X)$-morphism $N\rightarrow X$. 
\end{lem}
\begin{proof}
 Consider on $M$ the equivalence relation $x\sim y$ if $p(x)=p(y)$. By Lemma \ref{lem:dense_orbits}, the action of $\Gamma(M/N)$ has dense orbit in the fiber of $p$ then the equivalence classes of $\sim$ are the closure of orbits of $\Gamma(M/N)$. Furthermore, $\D$ is constant on the orbit of $\Gamma(M/N)$ and continuous hence constant on the equivalence classes of $\sim$. Therefore, $\D$ induces a continuous map $\sim\backslash M \rightarrow X$. Moreover, $p$ induces a continuous map $\sim\backslash M \rightarrow N$ which is open by Lemma \ref{lem:spread_connected_image}, injective by definition and surjective since $p$ is surjective; hence $\sim\backslash M \rightarrow N$ is an homeomorphism. We thus constructed a continuous map $\overline \D:N \rightarrow X$. Notice that $\overline \D$ is a $(G,X)$-morphism on $\ord(N)$ by usual results on $(G,X)$-manifolds. Finally, $\D$ is an a.e. $(G,X)$-morphism.
\end{proof}

\begin{prop} 
Let $(X_\alpha)_{\alpha\in A}$ be a family of  $(G,X)$-model spaces and let $M$ be a $X_A$-manifold $M$.
Assume that $\reg(M)$ is semi-locally simply connected  and that each $X_\alpha$ admits a developping map $\widetilde {X_\alpha} \rightarrow X$. Let $S\subset X$ be a skeletal subset such that $S\supset \sing(M)$.

Then, $\widetilde M^S \rightarrow M$ admits a developping map $\D_S:\widetilde M^S \rightarrow X$.

\end{prop}
\begin{proof}

  Let $\widetilde M\xrightarrow{p_M} M$ be the natural projection, let $m\in \sing(\widetilde M)$ and let $\U$ be a connected open chart neighborhood of $p(m)$ together with an embedding $\U\xrightarrow{\iota} X_\alpha$ for some $\alpha\in A$. Denote by $\widetilde X_\alpha \xrightarrow{p_\alpha} X_\alpha$ (resp. $\widetilde \U \xrightarrow{p_{\U}} \U$)  the universal covering of $X_\alpha$ (resp. of $\U$). Denote $\D_\alpha$ (resp. $\D$) the developping map of $X_\alpha$ (resp. of $M\setminus S = \reg(M)$). 
  Consider the restriction of $p_M$ to the connected component $\hat \U$ of $m$ in $p ^{-1}(\U)$, by Lemma \ref{lem:branched_covering_restriction} it is a branched covering and by Proposition \ref{prop:universality}  the covering $\widetilde \U\xrightarrow{p_\U} \U$ is universal.  There thus exists a branched covering $\widetilde \U \xrightarrow{f} \hat \U$ lifting the identity on $\U$. The embedding $\U\xrightarrow{\iota} X_\alpha$ lifts to a spread morphism $\ord(\widetilde \U) \rightarrow \ord(\widetilde X_\alpha)$ which, by Theorem \ref{theo:Fox_completion_functor}, extends to their completion to give a morphism of branched covering $\widetilde \U \xrightarrow{\widetilde \iota} \widetilde X_\alpha$.
  Note that $\widetilde \U \xrightarrow{f} \hat \U$  is quasi-Galoisian 
  of trivial holonomy; indeed a loop $\gamma$ in $\reg(\hat \U)$ of non trivial holonomy would embed into $\reg(\widetilde M)$ as a non trivial holonomy loop hence not homotopically trivial, this would contradict the simple connectedness of $\reg(\widetilde M)$. 
  The map $\D_\alpha\circ \widetilde \iota$ is a developping map of $\hat \U$ hence $\Gamma(\widetilde \U/\hat \U)$-invariant. 
  Then, by Lemma \ref{lem:descend_developing_map}, the a.e. $(G,X)$-morphism $\D_\alpha\circ \widetilde \iota$ induces an a.e. $(G,X)$-morphism $\hat \U \xrightarrow{\D_\U} X$. The diagram below sums up the situation.
 $$
 \xymatrix{
    X &&X&& X\\
    \widetilde X_\alpha\ar[d]^{p_\alpha}\ar[u]_{\D_{\alpha}}&\widetilde \U\ar@{-->}[l]^{\widetilde \iota}\ar@{-->}[r]\ar[d]^{p_\U}&\hat \U\ar@{-->}[u]^{\D_\U}\ar[d]^{p_{M|\hat \U}}\ar@{^{(}->}[r]&\widetilde M\ar[d]^{p_M}& \ar@{_{(}->}[l]\widetilde {M\setminus S}\ar[d]\ar[u]^{\D}\\ 
    X_\alpha& \ar@{_{(}->}[l]^{\iota}\U\ar@{=}[r]&\U\ar@{^(->}[r]&M&\ar@{_{(}->}[l]M\setminus S
}$$

  Since $\D_{\U|\reg(\hat\U)}$ and $\D_{|\reg(\hat \U)}$ are both $(G,X)$-map and both go from $\reg(\hat \U)$, which is connected, to $X$ there exists a unique $g\in G$ such that 
  $\D_{|\reg(\hat \U)} = g\circ \D_{\U|\reg(\hat \U)} $. We can thus continuously extend $\D$ to $\hat\U$ defining $\D(x)= g\circ \D_{\U}(x)$ for all $x\in \U$.
  Finally, $\D$ continuously extends to $\widetilde M$ and $\D:\widetilde M\rightarrow X$ is a developping map of  $\widetilde M\rightarrow M$. 
  
  Since $S\supset \sing(M)$ and by universality of $\widetilde M^S$, there exists a branched covering $\widetilde M^S \xrightarrow{\pi} \widetilde M$ which is an a.e. $(G,X)$-morphism with respect to the natural $(G,X)$-structure on $\widetilde M^S$.  The map $\D_S:=\D\circ \pi$ is then a developping map of $\widetilde M^S \rightarrow M$.

 \end{proof}

\section{Branched covering of Cauchy-compact $\mass{0}$-manifolds}

The aim to this section is to provide a geometric construction of the branched covering of a Cauchy-compact $\mass{0}$-manifold.

\subsection{The branched covering of the BTZ model space}

In \cite{BTZext}, the author introduced the BTZ model space $\mass{0}$ which is defined as $\RR^3$ endowed with the singular $\mass{}$-structure induced by the flat lorentzian metric $\d s^2=-2\d\tau \d \r +\d \r^2+ r^2 \d \theta$ in cylindrical coordinates.
The singular locus of $\mass{0}$ is then $\sing(\mass{0})=\{\r=0\}$ and the regular locus 
is $\reg(\BTZ):= \{\r>0\}$. 

By remark \ref{rem:simple_cases},  $\reg(\mass{0})$ is semi-locally simply connected at every point of $\mass{0}$. Cylindrical coordinates shows the regular locus is homeomorphic to $\RR\times \RR_+^*\times \RR/2\pi\ZZ$, and we can thus deduce that the completed spread of $\widetilde \reg(\mass{0})$  is the topological space
$$\mass{0,\infty} := (\RR\times \RR_+ \times \RR/2\pi \ZZ)/\sim$$
with $\sim$ the equivalence relation identifying two points $(\tau,\theta,\r)$ and $(\tau',\theta',\r')$ if and only if they are either equal or $\tau=\tau'$ and $\r=\r'=0$.
The topology being the quotient topology.
We notice that the developping map of $\reg(\mass{0})$ extends continuously to $\mass{0,\infty}$ to give the map

$$\fonction{\D}{\mass{0,\infty}}{\E^{1,2} }{\left(\tau,\r,\frac{\theta}{2\pi}\right)}{
\begin{pmatrix}t \\ x \\ y\end{pmatrix}=\begin{pmatrix} \tau + \frac{1}{2}\r\theta^2 \\ \tau + \frac{1}{2}\r\theta^2 -r \\ -\r\theta  \end{pmatrix}
}.$$

The image of this map is the  causal future of a lightlike line $\Delta$ ie $\D(\mass{0,\infty})=J^+(\Delta)$. More precisely, $\Delta$ is the line directed by $\overrightarrow u = (1,1,0)$ through the origin of Minkowski space, $\D(\mass{0,\infty})$ is then the union of $\Delta$ with the open half space of $\mass{}$ above the plane directed by $\overrightarrow u^\perp$ through the origin. We recover that the holonomy of $\mass{0}$ is parabolic, indeed the Galois group $\Gamma(\mass{0,\infty},\mass{0})$ is isomorphic to $\ZZ$ and stabilizes point-wise the singular locus; taking $\gamma$ a generator of $\Gamma(\mass{0,\infty},\mass{0})$, the holonomy $\rho$ sends $\gamma$ to $\phi:=\rho(\gamma)$ which thus stabilizes point-wise the image of $\sing(\mass{0,\infty})$ ie $\Delta$. Since $\Delta$ is lightlike, $\phi$ is parabolic.

Notice that the holonomy is faithful and that $\D$ is injective, therefore the developping map induces a bijective a.e. $\mass{}$-map 
$$\overline \D :  \mass{0,\infty} \xrightarrow{} J^+(\Delta)/\langle \gamma \rangle$$
which is however \textbf{not open} thus not a homeomorphism if the righthand side is endowed with the quotient topology of the usual topology of $J^+(\Delta)$. To force this map to be a homeomorphism, one has to add open subsets to $J^+(\Delta)$, which leads to the following definition.

\begin{defi}[BTZ topology]\label{defi:BTZtop} Let $\Delta$ be a lightlike line in $\mass{}$.
The BTZ topology on $J^+(\Delta)$ is the topology generated by the one induced by the natural topology of $\mass{}$ and open subsets of the form $I^+(p)\cup]p,+\infty[$  for  $p\in \Delta$ where $]p,+\infty[$ denote the relatively open future half-ray from $p$.
\end{defi}
\begin{prop} \label{prop:quotient_chart}
Let $\gamma \in \isom(\mass{})$ be a parabolic isometry and let $\Delta$ be its fixator in $\mass{}$. Any developping map $\D:\mass{0,\infty}\rightarrow \mass{}$  of holonomy $\gamma$ induces a isomorphism  $\overline \D : \mass{0}\rightarrow J^+(\Delta)/\langle \gamma\rangle$  if $J^+(\Delta)$ is endowed with the BTZ topology.
\end{prop}
\begin{proof}
Since $\overline \D$ it is continuous and bijective, it suffices to prove that $\D $ is open. The topology of $\mass{0}$ is generated by the open diamonds $Int(J^+(p)\cap J^-(q))$; 
if $p$ is a regular point then the image of such a diamond is an open diamond of $\mass{}$, otherwise the image is $I^+(\D(p))\,\cup\,]\D(p),+\infty[$ thus open for the BTZ topology.

\end{proof}

\subsection{The universal branched covering of $\mass{0}$-manifolds}

Consider $M$ a $\mass{0}$-manifold, ie a singular $\mass{}$-manifold whose singular locus is locally modelled on $\mass{0}$. The regular locus of such a manifold is a topological manifold, hence semi-locally simply connected, and the singular locus is a 1-dimensional closed submanifold hence its regular locus is semi-locally simply connected in $M$. Therefore by Proposition \ref{prop:Galoisian_semilocal}, $M$ admits a Galoisian universal covering $\widetilde M$ possibly branched over its singular locus. Furthermore, from the previous section, the singular model space admits a developping map, then so does $M$. We may add that the singular locus being a 1-submanifold of non trivial holonomy, $\widetilde M$ is exactly branched over $\sing(M)$.
We summarize this as follow.
\begin{prop} Let $M$ be a $\mass{0}$-manifold, then $M$ admits a Galoisian universal covering exactly branched above $\sing(M)$ and admits a developping map.
Furthermore, $\widetilde M$ is a $\mass{0,\infty}$-manifold.
\end{prop}

The goal of this section is to show, similarly to Cauchy-complete Cauchy-maximal $\mass{}$-manifolds, that Cauchy-complete Cauchy-maximal $\mass{0}$-manifolds can be realized as quotients of convex domains of Minkowski space. 

\begin{lem}\label{lem:increasing_map} Let $M$ be a  $\mass{A}$-manifold with $A=\{0,0\infty\}$ and $\D:M\rightarrow \mass{}$ an a.e. $\mass{}$-morphism. 
Then $\D$ is  increasing on $M$ and the causal order on $\mass{}$.
\end{lem}
\begin{proof} 
 The restriction of such a map to a chart neighborhood is the restriction of a developping map of the local model space. Since the developping maps of each model spaces are  increasing, $\D$ is locally increasing. Furthermore, since $M$ is connected, it is causally connected (ie for every $x,y\in M$ if $x\leq y$ there exists a future causal path from $x$ to $y$), therefore $\D$ is increasing.
 
\end{proof}
\begin{cor}\label{cor:increasing_map}
 For $M$ be a  $\mass{0}$-manifold, every developping map of $M$ is increasing.
\end{cor}
\begin{rem} It is noteworthy though not usuful here, that the universal branched covering of a $\mass{0}$-manifold is thus always causal.
\end{rem}

\begin{thm} Let $M$ be a globally hyperbolic Cauchy-complete Cauchy-maximal $\mass{0}$-manifold, let $\widetilde M$ its universal covering branched over $\sing(M)$, of  developping map and holonomy  $(\D,\rho)$. We note $\Gamma:=\Gamma(\widetilde M/M)\simeq\pi_1(\reg(M))$ its Galois group.  Then, up to time reversal of $M$:
\begin{itemize}
 \item $\D$ is injective
 \item $\rho$ is faithful, discrete and without torsion; 
 \item there exists a $\Gamma$-invariant family of relatively open future complete lightlike  rays $(\Delta_i)_{i\in I}$ and a $\Gamma$-invariant family of lightlike planes $(\Pi_j)_{j\in J}$ such that 
 $$\D(\widetilde M) = \left(\bigcap_{i\in I} J^+(\Delta_i) \right) \cup \left(\bigcap_{j\in I} I^+(\Pi_j) \right)$$
 
 \item endowing $\D(\widetilde M)$ with the BTZ topology in the neighborhood of the lightlike rays $(\Delta_i)_{i\in I}$, the developping map induces an isomorphism
 $$ M\xrightarrow{~~\overline D~~} \rho\backslash\D(\widetilde M
 ) $$
\end{itemize}

\end{thm}
\begin{proof}
 From Theorem 3 of \cite{BTZext}, the regular part of $M$ is globally hyperbolic, Cauchy-maximal and Cauchy-complete; therefore, by general result of Barbot \cite{barbot_globally_2004}, the restriction $\D_{|\reg(\widetilde M)}$ is injective, $\rho$ is faithful, discrete and without torsion and $$ \Omega:=\D(\reg(\widetilde M)) =\bigcap_{j\in I} I^+(\Pi_j)$$ for some $\Gamma$-invariant family $(\Pi_j)_{j\in J}$ of lightlike planes. Furthermore, $\Gamma$ acts totally discontinuously and freely on $\Omega$. Notice that $\Omega$ is future complete ie for all $x\in \Omega, J^+(x)\subset \Omega$.
 
 Note that from Corollary \ref{cor:increasing_map}, $\D$ is increasing.  
 Let $x\in \sing(\widetilde M)$, its stabilizer $\Gamma_x$ is isomorphic to $\ZZ$ and for $\gamma\in \Gamma_x\setminus\{1\}$, $\rho(\gamma)$ is parabolic. Therefore, the set of fix points of $\Gamma_x$ in $\mass{}$ is a lightlike line $\Delta$, the restriction of $\D$ to the BTZ line $\Delta$ through $x$ is then increasing  with image in $\Delta$ and since $\Delta_x$ is a totally ordered open future half line, so is its image $\D(\Delta_x)\subset\Delta$.
 Since $\Gamma$ does not act freely on $\Delta$  via $\rho$, $\Delta$ does not intersects $\Omega$ thus $\D(\Delta_x) \subset \partial\Omega$. 
 Finally, since $J^+(\D(\Delta_x)) = I^+(\D(\Delta_x))\cup \D(\Delta_x) $ and since $I^+(\D(\Delta_x))\subset \Omega$ we conclude that $\D(\widetilde M)$ has the wanted form. 
 
 Let $x,y\in \widetilde M$ such that $\D(x)=\D(y)$. We have $$\D(I^+(x))=I^+(\D(x))=I^+(\D(y))=\D(I^+(y))\subset \Omega$$
 and since $I^+(x)$ and $I^+(y)$ are subsets of $\reg(M)$ and $\D_{|\reg(M)}^{|\Omega}$ is a bijection we deduce that $I^+(x)=I^+(y)$. Since $\widetilde M$ is globally hyperbolic, it is future distinguishing (ie the function $ z\mapsto I^+(z)$ is injective) and $x=y$. Finally, $\D:\widetilde M \rightarrow \mass{}$ is injective.
 
 The map $\D_{|\reg(\widetilde M)}$ is open since it is a local homeomorphism. The image of a neighborhood basis of a singular point $x\in \sing(\widetilde M)$ is obtains by considering sets of the form $I_{p,q}:=\mathrm{Int}(J^+(p)\cap J^-(q))$ with $p\in J^-(x)$ and $q\in I^+(x)$. By taking $p,q$ sufficiently close to $x$, such a neighborhood can be chosen in a causally convex in $M$ chart neighborhood $\U$ of $x$.  This way $\mathrm{Int}(J^+(p)\cap J^-(q)) = \mathrm{Int}(J_\U^+(p)\cap J_\U^-(q))$ and the image by $\D$ of such a domain is the exactly $(I^+(\D(p))\cup \D(]p,+\infty[)\cap I^-(\D(q))$ with $]p,+\infty[$ the future half BTZ ray from $p$ in $\widetilde M$. Therefore, the image of such $\D(I_{p,q})$ together with the topology induced by $\mass{}$ generates the BTZ topology on $\D(\widetilde M)$.
 
 For the BTZ topology on $\D(\widetilde M)$, the map $\D : \widetilde M \rightarrow \D(\widetilde M)$ is then a $\Gamma$-invariant homeomorphism and induces a isomorphism of singular $\mass{}$-manifolds $\Gamma\backslash\widetilde M\simeq \rho\backslash\D(\widetilde M) $. Since $\widetilde M \rightarrow M$ is Galoisian, $M\simeq \Gamma\backslash\widetilde M$ and 
$M\simeq\rho\backslash\D(\widetilde M)$.

\end{proof}

%
%
%
\bibliographystyle{gtart}
\bibliography{note}	

\end{document}